%% file: main.tex
\documentclass{amsart}
\usepackage{verbatim}
\usepackage{enumerate}
\usepackage{MnSymbol}
\usepackage{enumitem}
\usepackage[mathscr]{euscript}

\newcommand{\mfrak}{\mathfrak{m}}
\newcommand{\F}{\mathbb{F}}
\newcommand{\rk}{\operatorname{rk}}
\newcommand{\pexp}{\operatorname{Exp}}
\newcommand{\Sa}[1]{\ensuremath{\mathscr{#1}}}
\newcommand{\Sh}[1]{\ensuremath{\mathscr{#1}^{\mathrm{Sh}}}}

\newcommand{\qalg}{\Q^\mathrm{alg}_p}

\newcommand{\Cal}[1]{\mathcal #1}

\def \dia {\diamondsuit}

\def \st {\operatorname{st}}

\def \Q{\mathbb{Q}}
\def \N{\mathbb{N}}

\def \sM{\mathscr{M}}
\def \sG{\mathscr{G}}
\def \sR{\mathscr{R}}
\def \sK{\mathscr{K}}

\def \sI{\mathscr{I}}
\def \sQ{\mathscr{Q}}
\def \R{\mathbb{R}}
\def \K{\mathbb{K}}

\def \Z{\mathbb{Z}}

\DeclareMathOperator{\dprk}{dp-rk}

\makeatletter

\setbox0\hbox{$\xdef\scriptratio{\strip@pt\dimexpr
    \numexpr(\sf@size*65536)/\f@size sp}$}

\newcommand{\myscriptarrow}[1][1cm]{{%
    \hbox{\rule[\scriptratio\dimexpr\fontdimen22\textfont2-.2pt\relax]
               {\scriptratio\dimexpr#1\relax}{\scriptratio\dimexpr.4pt\relax}}%
   \mkern-4mu\hbox{\let\f@size\sf@size\usefont{U}{lasy}{m}{n}\symbol{41}}}}

\newcommand{\nip}{\mathrm{NIP}}

\input{makros.tex}

\title{dp and other minimalities}
\author{Pierre Simon}
\address{Department of Mathematics, UC Berkeley, Berkeley CA 94720}
\email{pierre.simon@berkeley.edu}
\thanks{PS was partially supported by NSF (grants no. 1665491 and 1848562)}

\author{Erik Walsberg}
\address{Department of Mathematics\\University of Illinois at Urbana-Champaign\\1409 West Green Street\\Urbana, IL 61801}
\email{erikw@illinois.edu}
\urladdr{http://www.math.illinois.edu/\textasciitilde erikw}
\date{\today}
\begin{document}
\begin{abstract}
A first order expansion of $(\R,+,<)$ is dp-minimal if and only if it is o-minimal.
We prove analogous results for algebraic closures of finite fields, $p$-adic fields, ordered abelian groups with only finitely many convex subgroups (in particular archimedean ordered abelian groups), and abelian groups equipped with archimedean cyclic group orders.
The latter allows us to describe unary definable sets in dp-minimal expansions of $(\Z,+,S)$, where $S$ is a cyclic group order.
Along the way we describe unary definable sets in dp-minimal expansions of ordered abelian groups.
In the last section we give a canonical correspondence between dp-minimal expansions of $(\Q,+,<)$ and o-minimal expansions $\Sa R$ of $(\R,+,<)$ such that $(\Sa R,\Q)$ is a ``dense pair".
\end{abstract}
\maketitle

\noindent In model theory one typically approaches a first order structure $\mathscr{M}$ by showing $\mathscr{M}$ satisfies some form of model completeness which yields a description of definable sets highly specific to that structure.
This description allows one to understand the abstract classification-theoretic properties of $\mathscr{M}$ and situate $\mathscr{M}$ in the landscape of first order structures.
Converse implications, situations in which abstract classification-theoretic properties yield precise descriptions of definable sets, are rare.
One is Fact~\ref{fact:pierre}, a consequence of \cite[Corollary 3.7]{Simon-dp}.

\begin{Fact}
\label{fact:pierre}
An expansion of $(\mathbb{R},+,<)$ is dp-minimal if and only if it is o-minimal.
\end{Fact}

\noindent Another is Fact~\ref{fact:Z} \cite[Proposition 6.6]{ADHMS}.

\begin{Fact}
\label{fact:Z}
An expansion $\mathscr{Z}$ of $(\mathbb{Z},+,<)$ is dp-minimal if and only if every $\mathscr{Z}$-definable subset of every $\mathbb{Z}^n$ is already $(\mathbb{Z},+,<)$-definable.
\end{Fact}

\noindent
%We say that an expansion is \textbf{proper} if it adds new definable sets.
%Fact~\ref{fact:Z} generalizes to strongly dependent expansions \cite[Corollary 2.20]{DG}.
The initial motivation for this work was to prove an analogous result over $(\Q,+,<)$.
We prove several results of this form and give a general framework for such results.

\section{Dp-rank}
\noindent Our reference is \cite[Chapter 4]{Simon-Book}.
Let $\kappa$ be a cardinal and $T$ be a complete theory.
A family $(I_t : t \in X)$ of sequences of elements of a highly saturated $\sM \models T$ is \textit{mutually indiscernible} over a small set $A$ of parameters if each $I_t$ is indiscernible over $A \cup (I_s : s \in X \setminus \{t\} )$.
The dp-rank $\dprk(T)$ of $T$ is less than $\kappa$ if for every highly saturated $\sM \models T$, small set $A$ of parameters, family $(I_t : t < \kappa)$ of mutually indiscernible sequences over $A$, and $b \in \sM$ there is $\lambda < \kappa$ such that $I_\lambda$ is indiscernible over $Ab$.
The dp-rank of $T$ is $\kappa$ if $\dprk(T) < \kappa^+$ but $\dprk(T)$ is not less than $\kappa$.
The dp-rank of $T$ is $\infty$ if we do not have $\dprk(T) < \kappa$ for any cardinal $\kappa$.
It is known that $\dprk(T) < \infty$ if and only if $\dprk(T) < |T|^{+}$ if and only if $T$ is $\mathrm{NIP}$.
It is easy to see $\dprk(T) = 0$ if and only if $\mathscr{M}$ is finite.
We say $T$ is \textbf{dp-minimal} when $\dprk(T) \leq 1$.
A structure is dp-minimal if its theory is dp-minimal.
See \cite{DGl-dp, Simon-dp, AldE} for examples of dp-minimal structures.
Dp-minimal fields are classfied in \cite{Johnson}.

\section{Set up}

\subsection{Conventions}
\noindent
All languages and structures are first order and ``definable" without modification means ``first order definable, possibly with parameters".
Given languages $L \subseteq L'$ and an $L'$-structure $\sM$ we let $\sM|L$ be the $L$-reduct of $\sM$.
We let $|x| = n$ be the length of a tuple $(x_1,\ldots,x_n)$ of variables.
Throughout $n$ is a natural number and $m,k,i,j$ are integers.

\subsection{External definibility}
\noindent 
Suppose $\sM$ is a structure and $\mathscr{N}$ is an $|M|^{+}$-saturated elementary extension of $\sM$.
A subset of $M^k$ is \textbf{externally definable} if is of the form $X \cap M^k$ for $\mathscr{N}$-definable $X \subseteq N^k$.
The \textbf{Shelah expansion} $\mathscr{M}^{{\mathrm{Sh}}}$ of $\mathscr{M}$ is the expansion of $\mathscr{M}$ by a $k$-ary predicate defining $X \cap M^k$ for every $\mathscr{N}$-definable $X \subseteq N^k$.
Up to interdefinibility, this construction does not depend on choice of $\mathscr{N}$.
(Two structures on a common domain $M$ are \textbf{interdefinable} if they define the same subsets of all $M^n$).
Observe that if $\sM$ expands a linear order then every convex set is externally definable.
The following theorem of Shelah is deep~\cite{Shelah-strongly}.
(See \cite{ChSi-externally} for another proof.)

\begin{Fact}\label{fact:shelah}
If $\mathscr{M}$ is $\mathrm{NIP}$ then every $\mathscr{M}^{\mathrm{Sh}}$-definable subset of every $M^k$ is externally definable.
It follows that $\dprk(\mathscr{M}^{\mathrm{Sh}}) = \dprk(\mathscr{M})$ for any $\mathscr{M}$ and consequently $\mathscr{M}^{\mathrm{Sh}}$ is $\mathrm{NIP}$ when $\mathscr{M}$ is $\mathrm{NIP}$ and $\mathscr{M}^{\mathrm{Sh}}$ is dp-minimal when $\mathscr{M}$ is dp-minimal.
It follows that if $\sM$ expands a linear order then any expansion of $\sM$ by any family of convex sets is $\nip$ when $\sM$ is $\nip$ and dp-minimal when $\sM$ is dp-minimal.
\end{Fact}

\subsection{Weak minimality}
\noindent Let $L \subseteq L^{\dia}$ be languages, $T$ be a complete consistent $L$-theory, $T^\dia$ be a complete consistent $L^\dia$-theory extending $T$, $\sM^\dia \models T^\dia$, and $\sM = \sM^\dia | L$.
We say that $\mathscr{M}^\dia$ is \textbf{$\pmb{\mathscr{M}}$-minimal} if every $\mathscr{M}^\dia$-definable subset of $M$ is $\mathscr{M}$-definable.
Furthermore $T^\dia$ is \textbf{$\pmb{T}$-minimal} if for every $L^\dia$-formula $\varphi(x,y)$, $|x| = 1$ there is an $L$-formula $\psi(x,z)$ such that for every $\mathscr{O}^\dia \models T^\dia$ and $b \in O^{|y|}$ there is $c \in O^{|z|}$ such that
$$ \{ a \in O : \mathscr{O}^\dia \models \varphi(a,b) \} = \{ a \in O : \mathscr{O}^\dia \models \psi(a,c) \}.$$
\noindent
We say that $\mathscr{M}^\dia$ is \textbf{weakly $\pmb{\mathscr{M}}$-minimal} if every $\mathscr{M}^\dia$-definable subset of $M$ is externally definable in $ \mathscr{M}$.
We say that $T^\dia$ is \textbf{weakly $\pmb{T}$-minimal} if for every $L^\dia$-formula $\varphi(x,y)$ with $|x| = 1$, there is an $L$-formula $\psi(x,z)$ such that for every $\mathscr{O}^\dia \models T^\dia$, sufficiently saturated $\mathscr{P} \succ \mathscr{O}^\dia|L$, and $a \in O^{|y|}$ there is a $b \in P^{|z|}$ such that 
$$ \{ c \in O : \mathscr{O}^\dia \models \varphi(c,a) \} = \{ c \in O : \mathscr{P} \models \psi(c,b) \}.$$

\noindent
If $L$ is the empty language then $\mathscr{M}^\dia$ is $\mathscr{M}$-minimal if and only if $\mathscr{M}$ is minimal and $T^\dia$ is $T$-minimal if and only if $T^\dia$ is strongly minimal.
If $T$ is the theory of a dense linear order then $\mathscr{M}^\dia$ is $\mathscr{M}$-minimal if and only if $\mathscr{M}^\dia$ is o-minimal.
If $T$ is the theory of a $C$-relation, then $T^\dia$ is $T$-minimal if and only if $T^\dia$ is $C$-minimal.
If $T$ is the theory of $\Q_p$ then $T^\dia$ is $T$-minimal if and only if $T^\dia$ is $P$-minimal.
(See \cite{HaMa-Cell, HaMa-A} for an account of $C$-, $P$-minimality, respectively.)
If $\mathscr{M}$ is a dense linear order then $\Sa M^\dia$ is weakly $\Sa M$-minimal if and only if $\Sa M^\dia$ is weakly o-minimal and $T^\dia$ is weakly $T$-minimal if and only if $T^\dia$ is weakly o-minimal.
(An expansion of a dense linear order is \textbf{weakly o-minimal} if every definable unary set is a finite union of convex sets and a theory is weakly o-minimal if every model is.)
\medskip

\begin{comment}
\medskip \noindent Micheaux and Villemaire showed there are no proper $(\N,+)$-minimal expansions of $(\N,+)$~\cite{MiVi}.
Pillay and Steinhorn \cite{PiSt-discrete} showed there are no proper $(\mathbb{N},<)$-minimal expansions of $(\mathbb{N},<)$.
\end{comment}

\noindent
%We are only interested weak minimality in the case when $\sM$ is $\mathrm{NIP}$.
It is easy to see that Proposition~\ref{prop:weak-to-nip} is true.
We leave this to the reader.

\begin{Prop}
\label{prop:weak-to-nip}
Suppose $T^\dia$ is weakly $T$-minimal.
Then $T^\dia$ is $\nip$ when $T$ is $\nip$ and $T^\dia$ is dp-minimal when $T$ is dp-minimal.
\end{Prop}

\noindent
If every externally definable subset of $M$ is already definable in $\mathscr{M}$ then weak $\mathscr{M}$-minimality is equivalent to $\mathscr{M}$-minimality.
This condition is satisfied in the following situations: $\mathscr{M}$ is stable, $\mathscr{M}$ is an o-minimal expansion of $(\mathbb{R},<)$, $\mathscr{M}$ is $\mathbb{Q}_p$, or $\sM$ is $(\Z,+,<)$.
In each of the preceding examples $\sM^{\mathrm{Sh}}$ is interdefinable with $\sM$.
The first case holds as stability of a theory $T$ is equivalent to the assertion that all externally definable sets in all models of $T$ are definable.
The second is a consequence of the Marker-Steinhorn theorem \cite{MaSt-definable}, the third is a result of Delon \cite{Delon-def}, and the fourth follows by Fact~\ref{fact:shelah} and Fact~\ref{fact:Z}.
%(The fourth claim also follows from \cite{pres-types}).
\medskip

\noindent We seek dp-minimal $\mathscr{M}$ which satisfy the following for all expansions $\sM^\dia$ of $\sM$: 
\[
\textit{If $\mathscr{M}^\dia$ is dp-minimal then $\mathscr{M}^\dia$ is weakly $\mathscr{M}$-minimal,}
\]
or, stronger,
\begin{equation}
    \sM^\dia \textit{ is dp-minimal if and only if } \mathrm{Th}(\sM^\dia) \textit{ is weakly } T\textit{-minimal.}\tag{$\bigstar$}
\end{equation}
\noindent
Weak $\sM$-minimality is the strictest condition we can impose on definable unary sets as $\sM^{\mathrm{Sh}}$ is dp-minimal whenever $\sM$ is dp-minimal.
Outside of very special cases it does not seem possible to replace ``dp-minimal" with a weaker $\nip$-theoretic property in $(\bigstar)$.
For example if $\sM$ is a strongly dependent structure and $A$ is a dense (in the sense of \cite{BV-independent}) $\mathrm{acl}$-independent subset of $\sM$, then $(\sM,A)$ is strongly dependent by \cite{BV-independent} and almost never weakly $\sM^{\mathrm{Sh}}$-minimal.
It also seems that $(\bigstar)$ requires $\sM$ to have a certain amount of algebraic structure.
If $\sM$ is a linear order, more generally a tree, then the expansion of $\sM$ by \textit{all} subsets of $M$ is dp-minimal~\cite[Proposition 4.7]{Simon-dp} and clearly not $\sM^{\mathrm{Sh}}$-minimal.
%There are a variety of $\nip$ structures $\Sa M$ such that the expansion of $\Sa M$ by all unary sets is $\nip$, this phenomenon is orthogonal to $(\bigstar)$.
\medskip

\noindent
Theorem~\ref{thm:main} summarizes most of the results of this paper.
We emphasize that most examples of ordered abelian groups which arise away from ``valuational" settings have only finitely many convex subgroups.

\begin{Thm}\label{thm:main}
Suppose $\mathscr{M}$ is one of the following:
\begin{enumerate}
\item the algebraic closure of a finite field,
\item a non-singular ordered abelian group with only finitely many convex subgroups (in particular:
\begin{enumerate}
    \item a non-singular archimedean ordered abelian group,
    \item a non-singular subgroup of $(\R^n,+)$ with the lexicographic order, or
    \item a finite rank ordered abelian group), or
\end{enumerate}
\item a non-singular subgroup of $\R/\Z$ with the induced cyclic order.
\end{enumerate}
Then $\mathscr{M}^\dia$ is dp-minimal if and only if $\mathrm{Th}(\mathscr{M}^\dia)$ is weakly $\mathrm{Th}(\mathscr{M})$-minimal.
\end{Thm}

\noindent
An abelian group $A$ is \textbf{non-singular} if $|A/nA| < \aleph_0$ for all $n \geq 1$.
An ordered abelian group is dp-minimal if and only if it is non-singular, so the assumption of non-singularity simply ensures that weak $\mathrm{Th}(\sM)$-minimality implies dp-minimality.
\medskip

\noindent Item $(3)$ is a corollary to $(2a)$.
Our interest in $(3)$ is an application to the classification problem for dp-minimal expansions of $(\Z,+)$.
Every known proper dp-minimal expansion of $(\Z,+)$ defines either $<$, a $p$-adic valuation~\cite{AldE}, or a dense cyclic group order.
Walsberg and Tran \cite{cyclic-orders} show that $(\Z,+,C)$ is dp-minimal for any dense cyclic group order $C$.
The proof of $(3)$ characterizes dp-minimal expansions of such $(\Z,+,C)$ in terms of unary definable sets.
\medskip

\noindent
Fix a prime $p$.
We also prove a weaker $p$-adic result.
We let $\K$ be a finite extension of $\Q_p$ throughout.
Recall that the $p$-adic valuation on $\K$ is $\K$-definable \cite[Lemma 1.5]{belair-def}, so we do not distinguish between $\K$ as a field or as a valued field.
We also make use of the fact that the $p$-adic valuation $\K$ is discrete.
The field $\Q_p$ is shown to be dp-minimal in \cite[Corollary 7.8]{ADHMS} and, as the authors observe in the first paragraph of \cite[Section 7.2]{ADHMS}, their argument shows that $\K$ is dp-minimal.
(Hence all characteristic zero local fields are dp-minimal.)

\begin{Thm}
\label{thm:p-adics}
Any dp-minimal expansion of $\K$ is $\K$-minimal.
\end{Thm}

\noindent
Combining with Fact~\ref{fact:pierre} we see that if $K$ is a characteristic zero local field other than $\mathbb{C}$ then any dp-minimal expansion of $K$ is $K$-minimal.
This is sharp, if $v$ is a nontrivial valuation on $\mathbb{C}$ then $(\mathbb{C},v)$ is dp-minimal and not strongly minimal, hence not $\mathbb{C}$-minimal.
We do not know if the theory of every dp-minimal expansion of $\Q_p$ is $\mathrm{Th}(\Q_p)$-minimal.
Theorem~\ref{thm:main}.2.C shows that if $R$ is an archimedean real closed field then any dp-minimal expansion of $R$ is weakly o-minimal.
We prove the $p$-adic version of this result.
We first recall Fact~\ref{fact:p-adic}, which characterizes the relevant fields.

\begin{Fact}
\label{fact:p-adic}
Let $K$ be a field of characteristic zero.
Then the following are equivalent:
\begin{enumerate}
\item $K$ admits a discrete Henselian valuation $v$ with  finite residue field $\F$.
\item $K$ is isomorphic to an elementary subfield of a finite extension of $\Q_p$, where $p$ is the characteristic of $\F$.
\end{enumerate}
If (1) is satisfied then $v$ is $K$-definable.
\end{Fact}

\begin{proof}
It is easy to see that (2) implies (1).
Suppose (1) and let $v$ be such a valuation on $K$.
Then the completion of $(K,v)$ is a complete discretely valued field of characteristic zero with residue field $\F$.
Hence the completion is isomorphic to a finite extension $\K$ of $\Q_p$.
We therefore assume that $K$ is a subfield of $\K$ and $v$ is the induced valuation.
By \cite[Theorem 3.1]{prestel-roquette} $K$ is $p$-adically closed.
Model completeness for $p$-adically closed fields \cite[Theorem 5.1]{prestel-roquette} shows that $K$ is an elementary subfield of $\K$.
\end{proof}

\noindent
With this in place we state Theorem~\ref{thm:p-adic-1}.

\begin{Thm}
\label{thm:p-adic-1}
Suppose that $K$ is discretely Henselian valued field of characteristic zero with finite residue field.
Then any dp-minimal expansion of $K$ is weakly $K$-minimal.
\end{Thm}
\noindent
Ideally, results such as Theorems~\ref{thm:main}, \ref{thm:p-adics}, and \ref{thm:p-adic-1} should be corollaries to more general results.
For example Fact~\ref{fact:pierre} is a corollary to the theorem of Simon that an expansion of an ordered abelian group is o-minimal if and only if it is definably connected and dp-minimal.
Fact~\ref{fact:Z} generalizes to Theorem~\ref{thm:ZZ} below.
Theorem~\ref{thm:main}.1 is a corollary to the theorem of Johnson that a non-strongly minimal dp-minimal field admits a definable type V field topology.
The first claim of Theorem~\ref{thm:p-adics} is a corollary to a more general result on dp-minimal expansions of discretely valued fields, Theorem~\ref{thm:discrete-val} below.
Theorem~\ref{thm:main}.2 is a special case of Theorem~\ref{thm:main-cc}.

\begin{Thm}
\label{thm:main-cc}
Suppose $(M,+,<)$ is an ordered abelian group and $\sM$ is an expansion of $(M,+,<)$ which defines only finitely many convex subgroups.
Then $\sM$ is dp-minimal if and only if every unary definable set in every elementary extension $\mathscr{N}$ of $\sM$ is a finite union of sets of the form $C \cap (a + nN)$ for convex $C$ and $a \in N$.
\end{Thm}

\noindent
Theorem~\ref{thm:main-cc} follows from Theorem~\ref{thm:main-eq}.

\begin{Thm}
\label{thm:main-eq}
Suppose $(M,+,<)$ is an ordered abelian group, $\sM$ is a dp-minimal expansion of $(M,+,<)$, and $X$ is an $\sM$-definable subset of $M$.
Then there is $n$ and an $\sM$-definable convex equivalence relation $E$ on $M$ such that $E$ has only finitely many finite classes and for every $a \in M$ there is a finite $B \subseteq M$ such that $E_a \cap X = E_a \cap [\bigcup_{b \in B} b + nM]$, where $E_a$ is the $E$-class of $a$.
\end{Thm}

\noindent
Theorem~\ref{thm:main-cc} shows in particular that a dp-minimal expansion of $(\Q,+,<)$ is weakly o-minimal.
The reader might wonder if there are any proper weakly o-minimal expansions of $(\Q,+,<)$.
The o-minimal version of this question was at one point a subject of intense interest.
Wilkie~\cite{Wilkie-fusion} showed that there are proper o-minimal expansions of $(\Q,+,<)$.
In Section~\ref{section:dense-pairs} we describe a correspondence between weakly o-minimal expansions of $(\Q,+,<)$ and o-minimal expansions $\mathscr{R}$ of $(\R,+,<)$ such that $(\mathscr{R},\Q)$ is a ``dense pair".

\section{Algebraic closures of finite fields}
\noindent
Suppose $\mathbb{F}$ is the algebraic closure of a finite field and $\mathscr{F}$ is an expansion of $\mathbb{F}$.
Johnson~\cite{Jo-canonical} has shown that a dp-minimal expansion of a field is either strongly minimal or admits a definable non-trivial type V field topology.
A type V field topology is induced by an absolute value or valuation, see \cite[Appendix B]{EP-value}.
It is well-known that $\mathbb{F}$ does not admit a non-trivial absolute value or valuation.
We recall the proof.
Fix $a \in \mathbb{F}^{\times}$.
Then $a^n = 1$ for some $n$.
Suppose $v$ is a valuation on $\mathbb{F}$.
As $a^n = 1$ we must have $v(a) = 0$.
So $v$ is a trivial valuation.
Suppose $\|,\|$ is an absolute value on $\mathbb{F}$.
Then $\|a^n\| = \|1\| = 1$ and $\|a^n\| = \|a\|^n$, so we must have $\|a\| = 1$.
So $\|\hspace{3pt}\|$ is a trivial absolute value.
Strong minimality implies dp-minimality, hence:

\begin{Cor}
The expansion $\mathscr{F}$ is dp-minimal if and only if $\mathscr{F}$ is strongly minimal.
\end{Cor}

\noindent 
So $\mathscr{F}$ is dp-minimal if and only if $\mathrm{Th}(\mathscr{F})$ is $\mathrm{Th}(\mathbb{F})$-minimal.
By applying a recent results of Johnson we also obtain a somewhat weaker result for expansions of finite dp-rank.
We say that a structure is \textbf{dp-finite} if it has finite dp-rank.
Johnson has shown that an unstable dp-finite expansion of a field admits a definable V topology~\cite{2004.14732} and a stable dp-finite expansion of a field has finite Morely rank~\cite{1910.05932}.
Finite Morley rank implies finite dp-rank, so we obtain:

\begin{Cor}
The structure $\mathscr{F}$ is dp-finite if and only if $\mathscr{F}$ has finite Morley rank.
\end{Cor}

\noindent
We are not aware of any proper expansions of $\mathbb{F}$ of finite Morley rank.
For example it is still an open question if there is a nontrivial multiplicative subgroup $G$ of $\mathbb{F}$ such that $(\mathbb{F},G)$ has finite Morley rank.
If there are infinitely many $\mathrm{Char}(\mathbb{F})$-Mersenne primes then there is no such subgroup \cite[Theorem 4]{Wagner-Bad}.

\section{Discretely valued fields}
\label{section:discretely}
\noindent 
We first show that dp-minimal expansions of discretely valued fields are ``locally $P$-minimal".
In this section $K$ is a field, $v : K^\times \to \Z$ is a non-trivial discrete valuation on $K$, and $\sK$ is an expansion of $(K,v)$.
We let $B(a,k)$ be the ball $\{ b \in K : v(a - b) \geq k \}$, $B(k) := B(0,k)$, and $P_n := \{ a^n : a \in K^\times \}$ for all $n$.
Dp-minimal valued fields are Henselian~\cite{JSW}, so we suppose that $(K,v)$ is Henselian.

\begin{Thm}
\label{thm:discrete-val}
Suppose $\sK$ is dp-minimal, $X$ is a $\sK$-definable subset of $X$, and $a \in K$.
Then there is $k$ such that $B(a,k) \cap X = B(a,k) \cap Y$ where $Y$ is a finite union of sets of the form $a + b P_n$ for some $b \in K$ and $n$.
\end{Thm}

\noindent
Furthermore if $L$ is a dp-minimal field then the valued field $L((t))$ of Laurent series is dp-minimal~\cite{Johnson}.
In particular Theorem~\ref{thm:discrete-val} applies to dp-minimal expansions of $\K((t))$ for $\K$ a characteristic zero local field.
We gather some basic lemmas.

\begin{Fact}
\label{fact:dp-field}
Suppose $\sK$ is dp-minimal and $X$ is a $\sK$-definable subset of $K$.
Then $X$ is the union of a definable open set and a finite set.
Furthermore the closure of $X$ in $K$ is the union of $X$ together with a finite set.
\end{Fact}

\noindent
The first claim of Fact~\ref{fact:dp-field} is \cite[Corollary 4.7]{JSW}.
The second claim is an easy consequence of the first.

\begin{Fact}
\label{fact:group}
Let $(M,+)$ be a non-singular abelian group.
Then any finite index subgroup of $(M,+)$ is a finite union of cosets of $nM$ for some $n$.
\end{Fact}

\begin{proof}
Suppose $H$ is an index $m$ subgroup of $(M,+)$.
We show that $(m!)M$ is a subgroup of $H$ and it follows by non-singularity that $H$ is a finite union of cosets of $(m!)H$.
Fix $a \in M$.
There are $1 \leq i < j \leq m + 1$ such that $ia,ja$ lie in the same coset of $H$.
So $(j - i)a$ is in $H$.
As $j - i$ divides $m!$ we have $(m!)a \in H$.
\end{proof}

\noindent
Fact~\ref{fact:germ} is a special case of \cite[Lemma 3.5]{JSW}.
We say that $X,X' \subseteq K$ have the same germ at $a \in K$ if $B(a,k) \cap X = B(a,k) \cap X'$ for some $k$.
A family $\Cal X$ of subsets of $X$ has only finitely many germs at $a$ if there are $X_1,\ldots,X_n \in \Cal X$ such that every $X \in \Cal X$ has the same germ at $a$ as some $X_i$.

\begin{Fact}
\label{fact:germ}
Suppose $\sK$ is dp-minimal.
Then any $\sK$-definable family of subsets of $K$ has only finitely many germs at any $a \in K$.
\end{Fact}

\begin{Lem}
\label{lem:pgroup}
Suppose $K$ is dp-minimal.
Then $P_n$ is finite index in $K^\times$ for any $n \geq 1$ and any finite index subgroup of $K^\times$ is a finite union of cosets of some $P_n$.
\end{Lem}

\begin{proof}
The second claim follows from the first claim and Fact~\ref{fact:group}.
We prove the first claim.
Fix $n \geq 1$.
It is easy to see that $0$ is an accumulation point of $P_n$, so $0$ is an accumulation point of $aP_n$ for each $a \in K^\times$.
Now apply Fact~\ref{fact:germ}.
\end{proof}

\begin{comment}
\noindent
See [ref] or [ref] for Fact~\ref{fact:int-nwd}.

\begin{Fact}
\label{fact:int-nwd}
Suppose $\sK$ is dp-minimal and $X$ is a $\sK$-definable subset of $X$.
Then $X$ is the union of a definable open set and a finite set.
\end{Fact}
\end{comment}

\noindent
We now prove Theorem~\ref{thm:discrete-val}.

\begin{proof}
After possibly translating we suppose $a = 0$.
If $0$ is not an accumulation point of $X$ then we take $Y = \emptyset$.
Suppose $0$ is an accumulation point of $X$.
Let $G$ be the set of $a \in K^\times$ such that $B(k) \cap aX = B(k) \cap X$ for some $k$.
Then $G$ is a $\sK$-definable subgroup of $K^\times$, the local multiplicative stabilizer of $X$ at $0$.
By Fact~\ref{fact:germ} the family $( aX : a \in K^\times )$ has finitely many germs at $0$.
It follows that $|K^\times/G| < \aleph_0$.
Applying Lemma~\ref{lem:pgroup} we get $n$ such that $P_n \subseteq G$.
\medskip

\noindent
Fix $t \in K$ such that $v(t) = 1$, so $v(t^n) = n$.
Fix $r \in \Z$ such that $B(r) \cap t^nX = B(r) \cap X$.
Fix representatives $\beta_1,\ldots,\beta_l \in X$ of the cosets of $P_n$ such that $0$ is an accumulation point of each $\beta_i P_n \cap X$ and set $m_i := v(\beta_i) > r$ for all $i$.
We show that
$$ B(m_i) \cap \beta_i P_n \cap X = B(m_i) \cap \beta_i P_n \quad \text{for all } \quad 1 \leq i \leq l. $$
It follows that $X$ and $\beta_1 P_n \cup \ldots \cup \beta_l P_n$ have the same germ at $0$.
\medskip

\noindent
Fix $1 \leq i \leq l$ and set $m := m_i$ and $\beta := \beta_i$.
Now observe that multiplication by $t^n$ maps $\beta P_n$ to $\beta P_n$ bijectively and maps $v^{-1}(j) \cap X$ to $v^{-1}(j + n) \cap X$  bijectively for all $j \geq m$.
So it suffices to show that $v^{-1}(m) \cap \beta P_n \cap X = v^{-1}(m) \cap \beta P_n$.
\medskip

\noindent
We have $v^{-1}(m) \cap \beta P_n = \{ a\beta : a \in P_n, v(a) = 0 \}$.
We fix $a \in P_n$ such that $v(a) = 0$ and show that $a\beta \in X$.
As $a \in P_n$ and $v(a) = 0$ multiplication by $a$ maps $v^{-1}(k) \cap X$ to itself bijectively for all sufficiently large $k$.
Fix such a $k \geq m$ and such that $k \in n\Z + m$.
Let $j = k - m$ and observe that $t^j$  is a power of $t^n$.
So multiplication by $p^j$ gives a bijection between $v^{-1}(m) \cap \beta P_n \cap X$ and $v^{-1}(k) \cap \beta P_n \cap X$.
It follows that multiplication by $t^{-j} a t^{j} = a$ gives a bijection from $v^{-1}(m) \cap \beta P_n \cap X$ to itself.
So $a\beta$ is an element of $X$.
\end{proof}

\begin{Thm}
\label{thm:padic}
Let $\K$ be a finite extension of $\Q_p$ for some prime $p$.
Then any dp-minimal expansion of $\K$ is $\K$-minimal.
\end{Thm}

\begin{proof}
Let $V$ be the valuation ring of the $p$-adic valuation on $\K$, recall that $V$ is compact.
Suppose that $X \subseteq \K$ is definable in a dp-minimal expansion of $\K$.
As $\K = V \cup V^{-1}$ we may suppose that $X\subseteq V$.
Theorem~\ref{thm:discrete-val} shows that for every $a \in V$ there is $k_a$ such that $B(a,k_a) \cap X$ is $\K$-definable.
As $V$ is compact there is finite $A \subseteq V$ such that $\{ B(a,k_a) : a \in A \}$ covers $V$.
So $X = \bigcup_{a \in A} (B(a,k_a) \cap X)$ is $\K$-definable.
\end{proof}

\noindent
The proof of Theorem~\ref{thm:padic} applies compactness of the valuation ring in a crucial way and so does not extend to elementary extensions of $\sK$.
So we cannot show that $\sK$ is $P$-minimal.

\begin{Ques}
\label{ques:p-adic}
Is every dp-minimal expansion of $\Q_p$ P-minimal?
Equivalently: is the theory of every dp-minimal expansion of $\Q_p$ $\mathrm{Th}(\Q_p)$-minimal?
\end{Ques}

\noindent
It is easy to see that Question~\ref{ques:p-adic} is equivalent to: Suppose $\sQ$ is a dp-minimal expansion of $\Q_p$, $\sQ'$ is an elementary extension of $\sQ$, and $X$ is a $\sQ'$-definable clopen subset of the valuation ring of $\sQ'$.
Must $X$ be a finite union of balls?
\medskip

\noindent
Recall our standing assumption that $(K,v)$ is a discrete Henselian valued field.

\begin{Thm}
\label{thm:padic-arch}
Suppose that $K$ is characteristic zero with finite residue field and $\sK$ is a dp-minimal expansion of $K$.
Then $\sK$ is weakly $K$-minimal.
\end{Thm}

\noindent
Let $\qalg$ be the algebraic closure of $\Q$ in $\Q_p$.
Then $\qalg$ is $p$-adically closed and is hence an elementary subfield of $\Q_p$.
So any dp-minimal expansion of $\qalg$ is weakly $\qalg$-minimal.
We now prove Theorem~\ref{thm:padic-arch}.

\begin{proof}
By Fact~\ref{fact:p-adic} we may suppose that $K$ is an elementary subfield of a finite extension $\K$ of $\Q_p$, where $p$ is the residue characteristic of $(K,v)$.
Let $\Sa F$ be a sufficiently saturated elementary extension of $\Sa K$ with underlying field $F$.
By saturation we may suppose that $\K$ is an elementary subfield of $F$.
We first realize $\K$ as an $F^{\mathrm{Sh}}$-definable set of imaginaries.
Let $u : F^\times \to \Gamma$ be the $p$-adic valuation on $F$, so $u$ is $F$-definable.
Note that $\Z$ is the minimal non-trivial convex subgroup of $\Gamma$.
Hence $\Z$ is $F^\mathrm{Sh}$-definable.
Let $w : F^\times \to \Gamma/\Z$ be the composition of $u$ with the quotient $\Gamma \to \Gamma/\Z$.
We equip $\Gamma/\Z$ with a group order by declaring $a + \Z \leq b + \Z$ when $a \leq b$, so $w$ is an externally definable valuation on $K$.
Let $W$ be the valuation ring of $w$ and $\mfrak$ be the  maximal ideal of $W$.
Then $W$ is the set of $\alpha \in K$ such that $u(\alpha) \geq m$ for some $m \in \Z$ and $\mfrak$ is the set of $\alpha \in K$ such that $u(\alpha) > m$ for all $m \in \Z$.
It is easy to see that for every $\alpha \in W$ there is a unique $\alpha^* \in \K$ such that $\alpha - \alpha^* \in \mfrak$.
We identify $W/\mfrak$ with $\K$ so that the residue map $\st : W \to \K$ is the usual standard part map.
Then $\K$ is an $F^{\mathrm{Sh}}$-definable set of imaginaries.
Let $\pmb{\Sa K}$ be the structure induced on $\K$ by $\Sh F$.
By Fact~\ref{fact:shelah} $\Sh F$ is dp-minimal, it follows that $\pmb{\Sa K}$ is dp-minimal.
By Theorem~\ref{thm:padic} $\pmb{\Sa K}$ is $\K$-minimal.
\medskip

\noindent
Suppose $X$ is a $\Sa K$-definable set.
We show that $X$ is externally definable in $K$.
By Fact~\ref{fact:dp-field} we may suppose that $X$ is closed in $K$.
Let $X'$ be the subset of $L$ defined by any formula defining $X$ and let $Y = \st(X' \cap V)$.
An easy saturation argument shows that $Y$ is the closure of $X$ in $\K$, so $X = Y \cap K$.
Note that $Y$ is $\pmb{\Sa K}$-definable, hence $Y$ is $\K$-definable.
So $X$ is externally definable in $K$.
\end{proof}

\noindent
We discuss the relationship between $\Sa K$ and $\pmb{\Sa K}$ further in Section~\ref{section:p adic completion}.

\section{Ordered abelian groups with finitely many convex subgroups}
\label{section:convex}

\noindent 
Throughout this section $(M,+,<)$ is an ordered abelian group and $\sM$ is an expansion of $(M,+,<)$.
Section~\ref{section:convex} is devoted to the proof of Theorem~\ref{thm:cc}.

\begin{Thm}
\label{thm:cc}
Suppose that $\sM$ defines only finitely many convex subgroups of $(M,+,<)$.
Then $\sM$ is dp-minimal if and only if every unary definable subset of every elementary extension $\mathscr{N}$ of $\sM$ is a finite union of sets of the form $C \cap (a + nN)$ for convex $C \subseteq N$, $n$, and $a \in M$.
\end{Thm}

\noindent
In particular if $(M,+)$ is divisible and $\Sa M$ defines only finitely many convex subgroups then $\Sa M$ is weakly o-minimal if and only if $\Sa M$ is dp-minimal.
Theorem~\ref{thm:cc} and Fact~\ref{fact:dp-oag} below show that if $(M,+)$ is non-singular and $(M,+,<)$ has only finitely many convex subgroups then $\sM$ is dp-minimal if and only if $\mathrm{Th}(\sM)$ is weakly $\mathrm{Th}(M,+,<)$-minimal.
\medskip

\noindent
There are dp-minimal expansions of divisible ordered abelian groups which are not weakly o-minimal and hence fail the hypothesis of Theorem~\ref{thm:cc}.
In particular there are dp-minimal ordered fields which are not weakly o-minimal.
%For example the ordered field $\R((t))$ of Laurent series is dp-minimal but the set of squares is not a finite union of convex sets.
A weakly o-minimal field is real closed by~\cite{MaMA-weakly}, the ordered field $\R((\Gamma))$ of Hahn series with exponents in a non-singular abelian group $\Gamma$ is dp-minimal~\cite{Jo-canonical}, and $\R((\Gamma))$ is real closed if and only if $\Gamma$ is divisible, see \cite[Theorem 4.3.7]{EP-value}.
In particular the field $\R((t))$ of Laurent series is dp-minimal and not weakly o-minimal.
\medskip

\noindent
We first gather some results and definitions concerning ordered abelian groups.
\subsection{Preliminaries on ordered abelian groups}
We make frequent use of the elementary facts that $(M,+,<)$ is either dense or discrete, if $(M,+,<)$ is discrete then $(M,+,<)$ has a minimal positive element which we denote by $1_M$, and $\{ m1_M : m \in \Z \}$ is the minimal convex subgroup of $(M,+,<)$ which is not $\{0\}$.
A subset $X$ of $M$ is \textbf{convex} if whenever $a,a' \in X$ and $a < b < a'$ then $b \in X$.
The convex hull of $X \subseteq M$ is the set of $b \in M$ such that $a \leq b \leq a'$ for some $a,a' \in X$.
A subgroup of $(M,+)$ is non-trivial if it is not $\{0\}$ or $M$.
We will make use of the fact that the convex subgroups of $(M,+,<)$ form a chain under inclusion.
\medskip

\noindent
For our purposes a \textbf{cut} in $M$ is a downwards closed subset of $M$ which either does not have a supremum or is of the form $(\infty,a]$ for some $a \in M$.
We let $\Cal C(M)$ be the set of cuts in $M$, order $\Cal C(M)$ under inclusion, and equip $\Cal C(M)$ with the resulting order topology.
We identify each $a \in M$ with the cut $(\infty,a]$ so $\mathcal{C}(M)$ is the order-completion of $(M,<)$.
A cut $C$ \textbf{lies in} a convex $D \subseteq M$ if $D$ intersects both $C$ and $M \setminus C$.
A family of cuts is nowhere dense if it is nowhere dense in $\Cal C(M)$.

\medskip
\noindent 
A \textbf{convex equivalence relation} $E$ on $M$ is an equivalence relation with convex equivalence classes.
We let $E_b$ be the $E$-class of $b \in M$.
There is a canonical linear order on $M/E$ given by declaring $E_a < E_b$ when every element of $E_a$ is strictly less than every element of $E_b$.
If $H$ is a convex subgroup of $(M,+,<)$ then equivalence modulo $H$ is a convex equivalence relation and the linear order on $M/H$ is a group order.
So we always regard the quotient of $(M,+,<)$ by a convex subgroup as an ordered abelian group.
\medskip

\noindent
A \textbf{cnc subset} of $M$ is a set of the form $C \cap X$ for \textbf{c}onvex $C \subseteq M$ and a \textbf{c}oset $X$ of $nM$ for some $n$.
Observe that every cnc subset of $M$ is definable in $(M,+,<)^{\mathrm{Sh}}$.
We make use of the following fact, whose verification we leave to the reader.
%(Recall that $(M,+)$ is non-singular if $nM$ has finite index for all $n \geq 1$.)

\begin{Fact}
\label{fact:cc-boolean}
Suppose $(M,+)$ is non-singular.
Then the collection of finite unions of cnc subsets of $M$ forms a boolean algebra.
\end{Fact}

\noindent
See \cite{JSW} for the first claim of Fact~\ref{fact:dp-oag}.
The second claim follows from the first claim and Fact~\ref{fact:cc-boolean}

\begin{Fact}
\label{fact:dp-oag}
The ordered abelian group $(M,+,<)$ is dp-minimal if and only if $(M,+)$ is non-singular.
So if $(M,+,<)$ is dp-minimal then the collection of finite unions of cnc subsets of $M$ is a boolean algebra.
\end{Fact}

\noindent
We say that $(M,+,<)$ is s $\mathbf{n}$\textbf{-regular} if every interval of cardinality at least $n$ contains an $n$-divisible element and $(M,+,<)$ is regular if it is $n$-regular for all $n$.
If $(M,+,<)$ is dense then $(M,+,<)$ is $n$-regular if and only if $nM$ is dense in $M$.
It is a theorem of Robinson and Zakon~\cite{Ro-Za} that $(M,+,<)$ is regular if and only if $(M,+,<)$ is elementarily equivalent to an archimedean ordered abelian group.
For each $n$ we let $R_n$ be the set of $a \in M$ such that every interval in $[0,|a|]$ with cardinality at least $n$ contains an $n$-divisible element.
Note that each $R_n$ is $(M,+,<)$-definable.
Fact~\ref{fact:bel} shows that $R_n$ is the maximal $n$-regular convex subgroup of $(M,+,<)$.
See e.g. \cite[Theorem 3.1]{Belegradek} for a proof of Fact~\ref{fact:bel}.

\begin{Fact}
\label{fact:bel}
Each $R_n$ is a convex subgroup of $(M,+,<)$.
\end{Fact}

\noindent In general $R_n$ could be $\{0\}$.
Lemma~\ref{lem:bel-1} rules out this situation in our case.

\begin{Lem}
\label{lem:bel-1}
If $(M,+,<)$ is discrete or $|M/nM| < \aleph_0$ then $R_n \neq \{0\}$.
\end{Lem}

\begin{proof}
If $(M,+,<)$ is discrete then $\{ k1_M : k \in \Z\}$ is contained in $R_n$.
Now suppose that $(M,+,<)$ is dense and $|M/nM| < \aleph_0$.
Then $M$ is a finite union of cosets of $nM$ so $nM$ is somewhere dense.
Suppose $I$ is a nonempty open interval in which $nM$ is dense.
Fix $a \in I \cap nM$.
Let $\delta$ be a positive element of $M$ such that $J := (a - \delta, a + \delta)$ is a subset of $I$.
Then $(J \cap nM) - a$ is dense in $(-\delta,\delta)$ and $(J \cap nM) - a \subseteq nM$.
So $nM$ is dense in $(-\delta,\delta)$, so $(-\delta,\delta)$ is a subset of $R_n$.
\end{proof}

\noindent
We leave Lemma~\ref{lem:Rn-index} to the reader.

\begin{Lem}
\label{lem:Rn-index}
Suppose $I$ is an interval in $R_n$ which contains at least $n$ elements.
Then $I$ intersects $a + nR_n$ for every $a \in R_n$.
\end{Lem}

\subsection{Right to left}
We first prove the right to left implication of Theorem~\ref{thm:cc}.

\begin{Prop}
\label{prop:cc-converse}
Suppose that every unary definable set in every elementary extension of $\sM$ is a finite union of cnc sets.
Then $\sM$ is dp-minimal.
\end{Prop}

\begin{proof}
By \cite[Theorem 2.2]{point-wagner} $(M,+)$ is non-singular.
So $(M,+,<)$ is dp-minimal by Fact~\ref{fact:dp-oag}.
An application of compactness shows that for every $\sM$-definable family $(X_a)_{ a \in M^k}$ of subsets of $M$ there is $n$ such that every $X_a$ is a union of no more than $n$ sets of the form $C \cap (a + nM)$ for convex $C \subseteq M$ and $a \in M$. 
It is now easy to see that $\mathrm{Th}(\sM)$ is weakly $\mathrm{Th}(M,+,<)$-minimal.
So $\Sa M$ is dp-minimal by Proposition~\ref{prop:weak-to-nip}.
\end{proof}

\subsection{Unary definable sets in dp-minimal expansions of ordered abelian groups}
In this section we prove Theorem~\ref{thm:main-eq}.

\begin{Lem}
\label{lem:coset-intersect}
Suppose $G$ is a group, $\sG$ is a first order expansion of $G$, $H$ is an infinite definable subgroup of $G$, and $X$ is an infinite definable subset of $G$ such that $X \cap aH$ has at most one element for every $a \in G$.
Then $\sG$ is not dp-minimal.
\end{Lem}

\noindent
We construct an array witnessing failure of dp-minimality, see \cite[Definition 4.21]{Simon-Book} for an overview of such arrays.

\begin{proof}
Let $(a_i : i \in \N), (b_j : j \in \N)$ be sequences of distinct elements of $X,H$, respectively.
Let $\phi_1(x,y) := (x^{-1}y \in H)$ and $\phi_2(x,y) := (yx^{-1} \in X)$.
Fix $k,l \in \N$.
Then $\sG \models \phi_1(a_i, a_kb_l)$ if and only if $i = k$ and $\sG \models \phi_2(b_j, a_k b_l)$ if and only if $j = l$.
So $( \phi_1(a_i,y) : i \in \N)$ and $(\phi_2(b_j, y) : j \in \N )$ form an array witnessing $\dprk(\sG) \geq 2$.
\end{proof}

\begin{Lem}
\label{lem:coset-intersect-1}
Let $H$ be an infinite $\sM$-definable subgroup of $(M,+)$, $X \subseteq M$ be $\sM$-definable, and $|X \cap (a + H)| < \aleph_0$ for all $a \in M$.
If $\dprk(\sM) = 1$ then $X$ is finite.
\end{Lem}

\begin{proof}
Let $Y$ be the set of $a \in X$ such that $a$ is maximal in $X \cap (a + H)$.
Then $|Y \cap (a + H)| \leq 1$ for all $a \in M$.
Apply Lemma~\ref{lem:coset-intersect}.
\end{proof}

\noindent The first claim of Fact~\ref{fact:finite-index} is \cite[Lemma 3.2]{Simon-dp}.
Note that the convex hull of a subgroup of an ordered abelian group is a subgroup.

\begin{Fact}
\label{fact:finite-index}
If $\sM$ is dp-minimal then every definable subgroup of $(M,+)$ has finite index in its convex hull.
\end{Fact}

\noindent Lemma~\ref{lem:finite-index} shows that every definable subgroup is a finite union of cnc sets.

\begin{Lem}
\label{lem:finite-index}
Suppose $\sM$ is dp-minimal.
Let $H$ be an $\mathscr{M}$-definable subgroup of $(M,+)$.
Then $H$ is of the form $H' \cap C$ where $C$ is the convex hull of $H$ and $H'$ is a finite union of cosets of $nM$ for some $n$.
\end{Lem}

\begin{proof}
Let $C$ be the convex hull of $H$ in $M$.
So $C$ is a subgroup of $(M,+)$.
By Fact~\ref{fact:finite-index} $H$ has finite index in $C$.
By Fact~\ref{fact:group} $H$ is a finite union of cosets of $nC$ for some $n$.
Finally, it is easy to see that $nC + a = (nM + a) \cap C$ for any $a \in C$.
\end{proof}

\begin{Lem}
\label{lem:local-behavior}
Suppose $\sM$ is dp-minimal and one of the following is satisfied:
\begin{enumerate}
\item $(M,+,<)$ is dense, or
\item $(M,+,<)$ is discrete and $\sM$ is sufficiently saturated.
\end{enumerate}
Let $X \subseteq M$ be  $\sM$-definable.
There is $n$ such that every infinite interval $I$ contains an infinite interval $J$ such that $J \cap X = J \cap [\bigcup_{a \in A} a + nM]$ for finite $A \subseteq M$.
\end{Lem}

\noindent The proof of Lemma~\ref{lem:local-behavior} splits into the dense and discrete cases.
In the discrete case we apply Facts~\ref{fact:Z} and \ref{fact:shelah}.
In the dense case we apply results of Goodrick and Simon.
Fact~\ref{fact:goodrick-nwd} is \cite[Lemma 3.3]{goodrick}.

\begin{Fact}
\label{fact:goodrick-nwd}
If $(M,+,<)$ is dense and $\sM$ is dp-minimal then every nowhere dense definable subset of $M$ is finite.
\end{Fact}

\noindent
Fact~\ref{fact:translate} is proven in the proof of \cite[Thm 3.6]{Simon-dp}.
Note that while \cite[Thm 3.6]{Simon-dp} is stated for expansions of divisible ordered abelian groups the proof only uses this assumption in the final step.

\begin{Fact}
\label{fact:translate}
Suppose that $(M,+,<)$ is dense, $\Sa M$ is dp-minimal, and $X$ is an infinite definable subset of $M$.
Then there is $a \in X$ and an interval $I$ containing $0$ such that
\begin{enumerate}
\item $X^* = (X - a) \cap I$ is dense in $I$,
\item If $g,h \in X^*$ and $g + h \in I$ then $g + h, -g \in X^*$.
\end{enumerate}
\end{Fact}

\noindent
We now prove Lemma~\ref{lem:local-behavior}.

\begin{proof}
Applying non-singularity and either compactness (in the dense case) or saturation (in the discrete case) we see that it suffices to show that every infinite interval $I$ has an infinite subinterval $J$ such that $J \cap X = J \cap Y$ for some finite union $Y$ of cosets of some $nM$.
\medskip

\noindent
Let $I$ be an infinite open interval.
The case when $I \cap X$ is finite is trivial, so suppose that $I \cap X$ is infinite.
After replacing $X$ with $I \cap X$ if necessary we suppose $X \subseteq I$.
It now suffices to produce $p \in X$ and an infinite subinterval $J$ of $I$ containing $p$ such that $J \cap X = J \cap Y$ where $Y$ is a finite union of cosets of some $nM$.
\medskip

\noindent
Suppose $(M,+,<)$ is dense.
In this case every nonempty open interval is infinite.
After passing to an elementary extension of $\sM$ if necessary we suppose $\sM$ is sufficiently saturated.
By Fact~\ref{fact:translate} we suppose, after translating and shrinking $I$ if necessary, that $I$ contains $0$, $-I = I$, $X$ is dense in $I$, and if $g,h \in X$ satisfy $g + h \in I$ then $g + h, -g \in X$.
Saturation yields a convex subgroup $C$ contained in $I$.
Then $X \cap C$ is a subgroup of $(M,+)$ and $X$ is dense and hence cofinal in $C$.
Fact~\ref{fact:shelah} and Lemma~\ref{lem:finite-index} together show that $X \cap C$ is equal to $H \cap C$ where $H$ is a finite union of cosets of $nM$ for some $n$.
Now take $J \subseteq C$ to be any open interval containing $0$.
\medskip

\noindent
Now suppose $(2)$ holds.
We identify $1_M$ with $1$ and $1_M\Z$ with $\Z$, so $\Z$ is a convex subgroup of $(M,+,<)$.
Saturation implies $\Z$ is non-trivial.
Fact~\ref{fact:shelah} shows that $(\sM,\Z)$ is dp-minimal.
Lemma~\ref{lem:coset-intersect-1} shows that some coset of $\Z$ has infinite intersection with $X$.
After possibly translating and reflecting we suppose $\N \cap X$ is infinite.
This implies that $\N \subseteq I$.
Fact~\ref{fact:Z} yields $b \in \Z$, $n \geq 1$, and a finite union $Y$ of cosets of $nM$ such that 
$$ \{ a \in \Z : a > b\} \cap X = \{ a \in \Z : a > b\} \cap Y. $$
An application of saturation yields $b' \in I, b' > \N$ such that $(b,b') \cap X=(b,b') \cap Y$.
Take $J = (b,b')$ and note that $J$ is infinite.
\end{proof}

\begin{Lem}
\label{lem:eq-0}
Suppose $\sM$ is dp-minimal and one of the following holds:
\begin{enumerate}
\item $(M,+,<)$ is dense, or
\item $(M,+,<)$ is discrete and $\sM$ is sufficiently saturated.
\end{enumerate}
Let $E$ be a convex equivalence relation such that for every infinite interval $I$ we have $|I \cap E_b| \geq \aleph_0$ for some $b \in M$.
Then there are only finitely many finite $E$-classes.
\end{Lem}

\begin{proof}
Suppose $(M,+,<)$ is dense.
So an $E$-class is finite if and only if it is a singleton.
Let $Y$ be the set of $b \in M$ such that $E_b = \{b\}$.
We show that $Y$ is finite.
By Fact~\ref{fact:goodrick-nwd} it suffices to show that any nonempty open interval $I$ contains a nonempty interval $J$ such that $J \cap Y = \emptyset$.
Fix $a \in I$ such that $I \cap E_a$ is infinite.
As $I \cap E_a$ is convex and infinite there is a nonempty open interval $J \subseteq I \cap E_a$, so $J \cap Y = \emptyset$.
\medskip

\noindent
Suppose $(2)$ holds.
We identify $1_M$ with $1$ and $1_M\Z$ with $\Z$.
The $E$-class of $a \in M$ is infinite if and only if it is either contains an initial or final segment of $a + \Z$.
So $\{ a \in M : |E_a| < \aleph_0\}$ is $(\sM,\Z)$-definable.
Let $Y$ be the $(\sM,\Z)$-definable set of minimal elements of finite $E$-classes.
Suppose towards a contradiction that there are infinitely many finite $E$-classes.
\medskip

\noindent
First suppose there is $a \in M$ such that $a + \Z$ contains infinitely many finite $E$-classes.
After translating and reflecting if necessary suppose that $\N$ contains infinitely many finite $E$-classes.
Then $Y \cap \N$ is infinite.
Dp-minimality of $(\sM,\Z)$ and Fact~\ref{fact:Z} together yield $n \geq 1$, $b \in \N$, and a finite union $Z$ of cosets of $nM$ such that 
$$ \{ a \in \Z : a > b \} \cap Y = \{ a \in \Z : a > b \} \cap Z. $$
Applying saturation we obtain $b' \in M, b' > \N$ such that $(b,b') \cap Y = (b,b') \cap Z$.
So the $E$-class of each element of $(b,b')$ is finite, contradiction.
\medskip

\noindent
Now suppose that each coset of $\Z$ contains only finitely many finite $E$-classes.
Then $Y$ is infinite and $|Y \cap (a + \Z)| < \aleph_0$ for all $a \in M$.
This contradicts Lemma~\ref{lem:coset-intersect-1}.
\end{proof}

\begin{Thm}
\label{thm:gen-convex}
Suppose $\sM$ is dp-minimal and $X$ is an $\sM$-definable subset of $M$.
Then there is $n$ and an $\sM$-definable convex equivalence relation $E$ such that 
\begin{enumerate}
\item there are only finitely many finite $E$-classes,
\item every infinite $E$-class is open, and
\item  for every $a \in M$ we have $E_a \cap X = E_a \cap Y$ where $Y$ is a (necessarily finite) union of cosets of $nM$.
\end{enumerate}
\end{Thm}

\begin{proof}
Let $\Sa M \prec \Sa N$ be sufficiently saturated and $X'$ be the subset of $N$ defined by same formula as $X$.
Lemma~\ref{lem:local-behavior} yields $n$ such that every infinite interval $I \subseteq N$ has an infinite subinterval $J$ such that $J \cap X' = J' \cap Y$ for some union $Y$ of cosets of $nN$.
Let $R_n$ be as above.
Recall that $R_n$ is $(M,+,<)$-definable.
We define a binary relation $E$ on $M$.
We declare $(a,a) \in E$ for all $a$, and if $a < b$ then $a,b$ are $E$-related if and only if there are $a' < a$ and $b' > b$ such that:
\begin{enumerate}
\item $a',b'$ lie in the same coset of $R_n$, 
\item $(a',a)$ and $(b,b')$ both have at least $n$ elements, and
\item $(a',b') \cap X = (a',b') \cap Y$ for some union $Y$ of cosets of $nM$.
\end{enumerate}
Finally, if $a > b$ then $(a,b) \in E$ if and only if $(b,a) \in E$.
Note that $E$ is $\sM$-definable.
We now show that $E$ is an equivalence relation.
It suffices to show that $E$ is transitive.
Fix distinct $a,b,c \in M$ such that $a,b$ and $b,c$ are both $E$-equivalent.
There are several cases to consider.
Without loss of generality we suppose $a < b$, the case when $a > b$ is handled in the same way.
The cases when $a < c  < b$ and $c < a < b$ are easy.
\medskip

\noindent
We treat the case when $a < b < c$.
Fix $a' < a$, $b_0 < b < b_1$, $c < c'$ and finite unions $Y_0,Y_1$ of cosets of $nM$ such that $(a',a)$, $(b_0,b)$, $(b,b_1)$, and $(c,c')$ all have at least $n$ elements and $(b_0,c') \cap X = (b_0,c') \cap Y_0$, $(a',b_1) \cap X = (a',b_1) \cap Y_1$.
If $Y_0 = Y_1$ then $(a',c') \cap X = (a',c') \cap Y_1$ and $a,c$ are $E$-equivalent.
We therefore show that $Y_0 = Y_1$.
We first reduce to the case when $a,b,c$ are in $R_n$.
Note that $a',a,b_0,b,b_1,c,c'$ all lie in the same coset of $R_n$.
% Let $X^* = X - a$ and $E^*$ be the binary relation on $M$ where $(x,y) \in E^*$ if and only if $(x + a, y + a) \in E$.
% Note that if $x < y$ then $(x,y) \in E$ if and only if there are $x' < x$ and $y' > y$ such that
% \begin{enumerate}
% \item $x',y'$ lie in the same coset of $R_n$,
% \item $(x',x)$ and $(y',y)$ both have at least $n$ elements, and
% \item $(x',y') \cap X^* = (x',y') \cap Z$ for some union $Z$ of cosets of $nM$.
% \end{enumerate}
After replacing $X$ with $X^*$ and $E$ with $E^*$ we suppose that $a',a,b_0,b,b_1,c,c'$ are all in $R_n$.
The intersection of any coset of $nM$ with $R_n$ is a coset of $nR_n$ so it is now enough to show that $Y_0$ and $Y_1$ contain the same cosets of $nR_n$.
\medskip

\noindent
We show that every coset of $nR_n$ contained in $Y_0$ is contained in $Y_1$, the same argument yields the other direction.
Fix $d \in M$ and suppose that $d + nR_n \subseteq Y_0$.
As $(b_0,b)$ has at least $n$ elements Lemma~\ref{lem:Rn-index} shows that $d + nR_n$ intersects $(b_0,b)$.
Hence $d + nR_n$ intersects $Y_1$ which implies $d + nR_n \subseteq Y_1$.
\medskip

\noindent
We now show that $E$ is a convex equivalence relation.
Suppose $a < b$ are $E$-equivalent and fix $a < c < b$.
Fix $a' < a < b < b'$ such that $a',a,b,b'$ all lie in the same coset of $R_n$, $(a',a)$, $(b,b')$ both have at least $n$ elements, and $(a',b') \cap X = (a',b') \cap Y$ for some union $Y$ of cosets of $nM$.
Each coset of $R_n$ is convex so $c$ lies in the same coset of $R_n$ as $a',b'$.
Furthermore $(c,b')$ has at least $n$ elements and $(a',b') \cap X = (a',b') \cap Y$.
So $a,c$ are $E$-equivalent.
\medskip

\noindent
We show that every infinite $E$-class is open.
This is trivially true when $(M,+,<)$ is discrete so suppose $(M,+,<)$ is dense.
Then any nonempty open interval is infinite.
Suppose $E_a$ is infinite.
After possibly reflecting there is $b > a$ such that $b \in E_a$.
Fix $a' < a$ and $b' > b$ such that $a',b'$ lie in the same coset of $R_n$ and $(a',b') \cap X = (a',b') \cap Y$ for some union $Y$ of cosets of $nM$.
It is easy to see that if $a' < c < b'$ then $c \in E_a$ so $(a',b') \subseteq E_a$.
\medskip

\noindent
It remains to show that there are only finitely many finite $E$-classes.
Let $E'$ be the convex equivalence relation on $N$ defined by the same formula as $E$.
It suffices to show that $E'$ has only finitely many finite classes.
By Lemma~\ref{lem:eq-0} it suffices to fix an infinite interval $I \subseteq N$ and show that $E'_a \cap I$ is infinite for some $a \in I$.
Applying Lemma~\ref{lem:local-behavior} we obtain an infinite subinterval $J$ of $I$ such that $J \cap X' = J \cap Y$ for some union $Y$ of cosets of $nN$.
Applying saturation we obtain elements $a' < a < b < b'$ of $J$ such that $(a,b)$ is infinite, $(a',a)$, $(b,b')$ both have at least $n$ elements, and $a',b'$ lie in the same coset of $R_n$.
So $a,b$ are $E$-equivalent and $E_a$ contains $(a,b)$ and is therefore infinite.
\end{proof}

\noindent Theorem~\ref{thm:gen} is a corollary to Theorem~\ref{thm:gen-convex}.
When $(M,+)$ is divisible Theorem~\ref{thm:gen} shows that any $\sM$-definable subset of $M$ is the union of an open set and a finite set, this is already proven in \cite{Simon-dp}.

\begin{Thm}\label{thm:gen}
Suppose $\mathscr{M}$ is dp-minimal, $(M,+,<)$ is dense, and $X$ is an $\sM$-definable subset of $M$.
Then $X$ is a finite union of sets of the form $(nM + a) \cap U$ where $U$ is $\sM$-definable and open.
\end{Thm}

\begin{proof}
When $X$ is finite let $n = 0$ and $U = M$.
Suppose $X$ is infinite.
Let $n$ and $E$ be as in Theorem~\ref{thm:gen-convex}.
Let $A$ be a set of representatives of the cosets of $nM$.
Let $F := \{ a \in M : E_a = \{a\} \}$.
So $F$ is the set of $a \in M$ with finite $E$-class and $F$ is finite.
For each $a \in A$ let $U_a$ be the union of all $b \in M$ such that $b \notin F$ and $(a + nM) \cap X \subseteq E_b \cap X$.
Each $U_a$ is a union of infinite $E$-classes and is therefore open.
We have
$$ X = (F \cap X) \cup \bigcup_{a \in A} [U_a \cap (a +nM)]. $$
\end{proof}

\begin{comment}
\begin{proof}
If $I$ is a nonempty open interval then an application of Lemma~\ref{lem:local-behavior} to $I \cap X$ yields a nonempty open subinterval $J \subseteq I$ such that $J \cap X = J \cap H$ where $H$ is a finite union of cosets of $nM$ for some $n$.
Applying compactness and non-singularity of $(M,+)$ we obtain integers $n_1,\ldots,n_k$ and $a_1,\ldots,a_k \in M$ such that every nonempty open interval $I$ contains a nonempty open interval $J$ such that $J \cap X = J \cap \bigcup_{i \in A} [n_iM + a_i]$ for some $A \subseteq \{1,\ldots,k\}$.
Let $n = n_1n_1 \ldots n_k$.
Then every $n_iM$ is a finite union of cosets of $nM$.
So there are $b_1,\ldots,b_k \in M$ such that every nonempty open interval $I$ contains a nonempty open subinterval $J$ such that $J \cap X = J \cap \bigcup_{i \in A} [nM + b_i]$ for some $A \subseteq \{1,\ldots,k\}$.

For each $A \subseteq \{1,\ldots,k\}$ let $U_A$ be the union of all nonempty open intervals $J$ such that $J \cap X = J \cap \bigcup_{i \in A} [nM + b_i]$.
Note that each $U_A$ is open.
Furthermore 
$$U_A \cap X = U_A \cap \bigcup_{i \in A} [nM + b_i] = \bigcup_{i \in A} U_A \cap [ nM + b_i]$$
for each $A \subseteq \{1,\ldots,k\}$.
Finally, the preceding paragraph shows that $\bigcup_{A \subseteq \{1,\ldots,k\}} U_A$ is dense in $M$, so the set of $p \in M$ which do not lie in any $U_A$ is nowhere dense and hence finite by \cite[Lemma 3.3]{goodrick}.
\end{proof}
\end{comment}

\subsection{Proof of Theorem~\ref{thm:cc}}
We prove the left to right implication of Theorem~\ref{thm:cc}.
A cut $C$ is \textbf{valuational} if $C + a = C$ for some positive $a \in C$, and \textbf{non-valuational} otherwise.
%By convention $M$ is non-valuational.
An $\sM$-definable convex subgroup of $(M,+,<)$ is valuational if it contains an $\sM$-definable non-trivial convex subgroup, and non-valuational otherwise.
We say that $\sM$ is valuational if there is a non-trivial $\sM$-definable convex subgroup, and non-valuational otherwise.

\begin{Lem}\label{lem:nonval}
Suppose $H$ is an $\sM$-definable convex subgroup of $(M,+,<)$.
Then $H$ is valuational if and only if there is an $\sM$-definable valuational cut $C$ which lies in some coset of $H$.
In particular $\sM$ is valuational if and only if there is an $\sM$-definable valuational cut.
\end{Lem}

\begin{proof}
Suppose $H$ is valuational.
Let $K$ be a non-trivial definable convex subgroup of $H$ and let $C$ be the downwards closure of $K$.
Then $C + a = C$ for any positive element $a \in K$.
So $C$ is valuational.
Observe that $C$ lies in $H$.
\medskip

\noindent
Suppose $C$ is a definable valuational cut which lies in $b + H$.
It is easy to see that 
$ \{ a \in M : C + a = C = C - a \} $
is a non-trivial definable convex subgroup of $H$.
\end{proof}

\noindent
If $(M,+,<)$ is non-archimedean then there are $a,b > 0$ such that $na < b$ for all $n$, in which case the convex hull of $a\Z$ is a nontrivial convex subgroup of $(M,+,<)$.
Conversely, if $K$ is a convex subgroup of $(M,+,<)$, and $a,b$ are positive elements of $M$ such that $a \in K, b \notin K$, then $na < b$ for all $n$.
Hence $(M,+,<)$ is archimedean if and only if $(M,+,<)$ does not have a non-trivial convex subgroup, equivalently if every cut in $M$ is valuational.
Thus the property of being non-valuational is a definable analogue of the archimedean property.

\begin{Lem}
\label{lem:discrete-max}
Suppose $(M,+,<)$ is discrete and $C$ is a cut in $M$.
Then $C$ is non-valuational if and only if $C$ has a maximal element.
Furthermore $\sM$ is non-valuational if and only if every nonempty definable subset of $M$ which is bounded above has a maximal element.
\end{Lem}

\begin{proof}
Suppose $a$ is maximal in $C$.
Then $a + b \notin C$ for any positive $b \in M$, so $C$ is non-valuational.
Suppose $C$ is non-valuational.
Then there is $a \in C$ such that $a + 1_M \notin C$.
Then $a$ is maximal in $C$ as $a + 1_M$ is the successor of $a$.
\medskip

\noindent
We now treat the second claim.
Suppose $C$ is a valuational cut in $\sM$ and $a \in M_{>0}$ satisfies $b + a \in C$ for any $b \in C$.
As $b < b + a$ for all $b \in C$, $C$ does not have a maximum.
Suppose $\sM$ is non-valuational.
Suppose $X$ is a nonempty definable subset of $M$ which is bounded above.
Let $C$ be the downwards closure of $X$.
Then $C$ has a maximal element, which is the maximal element of $X$.
\end{proof}

\begin{Prop}
\label{prop:nonval1}
Suppose $\mathscr{M}$ is dp-minimal and $H$ is an $\sM$-definable convex subgroup.
Then the following are equivalent:
\begin{enumerate}
\item $H$ is valuational,
\item there is an $\sM$-definable convex equivalence relation $E$ such that every coset of $H$ contains infinitely many infinite $E$-classes,
\item there is an $\sM$-definable convex equivalence relation $E$ with infinitely many $E$-classes $E_a$ such that $|E_a| \geq \aleph_0$ and $E_a$ does not contain $a + H$.
\end{enumerate}
In particular $\mathscr{M}$ is valuational if and only if there is an $\mathscr{M}$-definable convex equivalence relation with infinitely many infinite equivalence classes.
\end{Prop}

\noindent
Proposition~\ref{prop:nonval1} requires Lemma~\ref{lem:eq-nwd} and Fact~\ref{fact:goodrick}.
We leave the easy proof of Lemma~\ref{lem:eq-nwd} to the reader.

\begin{Lem}
\label{lem:eq-nwd}
Suppose $(M,+,<)$ is dense.
Suppose $\Cal E$ is a collection of pairwise disjoint infinite convex subsets of $M$ which are bounded above.
Let $\Cal C$ be the collection of downwards closures of elements of $\Cal E$.
Then $\Cal C$ is a nowhere dense family of cuts.
\end{Lem}

\begin{comment}
\begin{proof}
Let $I$ be a nonempty open interval.
We show that $I$ contains a nonempty open interval $J$ such that no element of $\Cal C$ lies in $J$.
If every $Y \in \Cal E$ is disjoint from $I$ then no $C \in \Cal C$ lies in $I$.
So suppose that $Y \in \Cal E$ intersects $I$.
Then $Y \cap I$ is an infinite convex subset of $M$ so there is a non-empty open interval $J$ contained in $Y \cap I$.
Any $L \in \Cal E \setminus \{Y\}$ is disjoint from $J$ so the downwards closure of such an $X$ is either disjoint from or contains $J$.
\end{proof}
\end{comment}

\noindent We now recall a result of Goodrick~\cite[Lemma 3.3]{goodrick} which  generalizes Fact~\ref{fact:goodrick-nwd}.

\begin{Fact}\label{fact:goodrick}
Suppose $\mathscr{M}$ is dp-minimal and $(M,+,<)$ is dense.
Let $\mathcal{C}$ be an $\mathscr{M}$-definable family of cuts.
If $\mathcal{C}$ has infinitely many non-valuational elements then $\mathcal{C}$ is somewhere dense.
\end{Fact}

\noindent
We now prove Proposition~\ref{prop:nonval1}.

\begin{proof}
If $K$ is a non-trivial $\sM$-definable convex subgroup of $H$ then equivalence modulo $K$ is a convex equivalence relation and every coset of $H$ contains infinitely many classes.
So $(1)$ implies $(2)$.
\medskip

\noindent
It is clear that $(2)$ implies $(3)$.
We show that $(3)$ implies $(1)$.
Suppose $E$ satisfies the conditions of $(3)$.
If $E_a$ does not contain $a + H$ then $E_a$ is disjoint from either an initial or final segment of $a + H$.
After reflecting if necessary we suppose that there are infinitely many $E$-classes $E_a$ such that $E_a$ is infinite and does not contain a final segment of $a + H$.
Let $\Cal B$ be the family of downwards closures of such $E$-classes.
Note that every element of $\Cal B$ lies in some coset of $H$.
\medskip

\noindent
Suppose $(M,+,<)$ is dense.
Then an $E$-class is infinite if and only if it has at least two elements, so $\Cal B$ is an $\sM$-definable family of cuts.
Lemma~\ref{lem:eq-nwd} shows that $\Cal B$ is nowhere dense.
So $\Cal B$ contains a valuational cut by Fact~\ref{fact:goodrick}.
Lemma~\ref{lem:nonval} now shows that $H$ is valuational.
\medskip

\noindent
Now suppose $(M,+,<)$ is discrete.
We identify the minimal positive element of $M$ with $1$ and consider $\Z$ as a convex subgroup of $(M,+,<)$.
Then $(\sM,\Z)$ is dp-minimal by Fact~\ref{fact:shelah}.
The downwards closure of $E_a$ is in $\Cal B$ if there is $b \in E_a$ such that $E_a$ contains an initial segment of $b + \Z$ but does not contain $b + H$.
So $\Cal B$ is an $(\sM,\Z)$-definable family of cuts.
Suppose $H$ is non-valuational.
Lemma~\ref{lem:nonval} implies that each cut in $\Cal B$ is non-valuational.
Lemma~\ref{lem:discrete-max} shows that each cut in $\Cal B$ has a maximal element.
Let $Y$ be the $(\sM,\Z)$-definable set of maximal elements of cuts in $\Cal B$.
%If $a \in Y$ then
%$$(a + \Z) \cap E_a = (a + \Z) \cap \{ b \in M : b \leq a \}. $$
Note that distinct elements of $Y$ lie in distinct cosets of $\Z$.
As $Y$ is infinite Lemma~\ref{lem:coset-intersect} implies $(\sM,\Z)$ is not dp-minimal, contradiction.
\end{proof}

\begin{comment}
\begin{Lem}\label{lem:nonval3}
Suppose $(M,+,<)$ is dense.
If every $\sM$-definable convex equivalence relation has only finitely many infinite classes then every $\sM$-definable open subset of $\mathscr{M}$ is a finite union of convex sets.
It follows that if $\sM$ is dp-minimal and non-valuational then every $\sM$-definable open subset of $M$ is a finite union of convex sets.
\end{Lem}

\begin{proof}
Suppose every $\sM$-definable convex equivalence relation has only finitely many infinite classes.
Suppose $U \subseteq M$ is definable and open.
We define a convex equivalence relation $E$ on $M$ in the following way:
\begin{enumerate}
\item if one of $a,b$ is not in $U$ then $(a,b) \in E$ if and only if $a = b$, and
\item if $a,b \in U$ then $(a,b) \in E$ if and only if the closed interval with endpoints $a,b$ is a subset of $U$.
\end{enumerate}.
Then the $E$-class of $a \in M$ is infinite if and only if $a \in U$ and $U$ is the union of the infinite $E$-classes.
\end{proof}
\end{comment}

\noindent
Proposition~\ref{prop:going-down} in particular shows that if $\sM$ is non-valuational then every definable subset of $M$ is a finite union of cnc sets.

\begin{Prop}
\label{prop:going-down}
Suppose $\sM$ is dp-minimal and $H$ is an $\sM$-definable non-valuational convex subgroup of $(M,+,<)$.
Let $N := M/H$ and $\pi$ be the quotient map $M \to N$.
Then every $\sM$-definable subset of $M$ is a finite union of cnc sets and sets of the form $\pi^{-1}(Y) \cap (b + nM)$ for $\sM$-definable $Y \subseteq N$, $b \in M$, and $n$.
\end{Prop}

\noindent
Taking $H = M$ we see that if $\Sa M$ is non-valuational then every definable unary set is a finite union of cnc sets.

\begin{proof}
Let $X$ be an $\sM$-definable subset of $M$.
Let $n$ and $E$ be as obtained in Theorem~\ref{thm:gen-convex}.
Let $A$ be a set of representatives of the cosets of $nM$.
There are only finitely many finite $E$-classes, so all but finitely many cosets of $H$ intersect an infinite $E$-class.
Proposition~\ref{prop:nonval1} implies that all but finitely many cosets of $H$ are contained in some $E$-class.
Let $F \subseteq N$ be finite such that if $c \in N \setminus F$ then $\pi^{-1}(c)$ is contained in an $E$-class.
So if $c \in N \setminus F$ then
$$ \pi^{-1}(c) \cap X = \pi^{-1}(c) \cap  \bigcup_{a \in B} a + nM  $$
for some $B \subseteq A$.
For each $b \in A$ let $Y_b$ be the $\sM$-definable set of $c \in N \setminus F$ such that 
$$ \pi^{-1}(c) \cap [ b + nM ] \subseteq \pi^{-1}(c) \cap X . $$
%Note that every element of $(M/H) \setminus F$ is contained in some $Y_b$.
Then
$$ X \setminus \pi^{-1}(F) =  \bigcup_{b \in A}  \pi^{-1}(Y_b) \cap [b + nM ]. $$
So it is enough to show that $X \cap \pi^{-1}(c)$ is a finite union of cnc sets for each $c \in F$.
Fix $c \in F$.
As $H$ is non-valuational and there are only finitely many finite $E$-classes Proposition~\ref{prop:nonval1} shows that only finitely many $E$-classes intersect $\pi^{-1}(c)$.
For each $b \in A$ let $\Cal C_b$ be the set of all $E_a$ such that $E_a$ intersects $\pi^{-1}(c)$ and $b + nM \subseteq E_a \cap X$.
Then
$$ \pi^{-1}(c) \cap X = \bigcup_{b \in A} \bigcup_{C \in \Cal C_b} C \cap (b + nM). $$
So $\pi^{-1}(c) \cap X$ is a finite union of cnc sets.
\end{proof}

\noindent We now finally prove Theorem~\ref{thm:cc}.
%We only treat the first claim as the second claim is immediate from the first.

\begin{proof}
Suppose $\sM$ defines exactly $m$ non-trivial convex subgroups of $(M,+,<)$.
We apply induction on $m$.
If $m = 0$ then $\sM$ is non-valuational so this case follows from Proposition~\ref{prop:going-down}.
Suppose $m \geq 1$.
Let $H$ be the intersection of all non-trivial definable convex subgroups of $(M,+,<)$.
Then $H$ is the minimal non-trivial definable convex subgroup of $(M,+,<)$, so $H$ is non-valuational.
Let $N := M/H$.
If $K$ is a convex subgroup of $N$ then $\pi^{-1}(K)$ is a convex subgroup of $M$ containing $H$.
Note that there are exactly $m - 1$ non-trivial definable convex subgroups of $N$.
\medskip

\noindent
By Proposition~\ref{prop:going-down} and Fact~\ref{fact:cc-boolean} it suffices to suppose that $Y \subseteq N$ is definable and show that $\pi^{-1}(Y)$ is a finite union of cnc sets.
Induction yields a finite collection $\Cal Y$ of cnc subsets of $N$ such that $Y = \bigcup \Cal Y$.
As $\pi^{-1}(Y) = \bigcup_{Z \in \Cal Y} \pi^{-1}(Z)$ we may suppose that $Y$ is a cnc subset of $N$.
Let $Y = C \cap (a + nN)$ for convex $C \subseteq N$.
So $\pi^{-1}(Y) = \pi^{-1}(C) \cap \pi^{-1}(a + nN)$.
As $\pi$ is monotone $\pi^{-1}(C)$ is convex.
Fix $a' \in M$ such that $\pi(a') = a$.
Then $\pi^{-1}(a + nN) = a' + \pi^{-1}(nN)$.
Then $\pi^{-1}(nN)$ is a subgroup of $(M,+)$ containing $nM$.
By non-singularity $|M/nM| < \aleph_0$ so $\pi^{-1}(nN)$ is a finite union of cosets of $nM$.
So $\pi^{-1}(a + nN)$ is a finite union of cosets of $nM$.
\end{proof}

\noindent
We now gather some corollaries to Theorem~\ref{thm:cc}.

\subsection{Concrete examples} We first discuss some examples of ordered abelian groups with only finitely many convex subgroups.

\medskip\noindent Let $<_{\mathrm{Lex}}$ be the lexicographic order on $\R^n$.
It is easy to see that if $M$ is a subgroup of $(\R^n,+)$ then every convex subgroup of $(M,+,<_{\mathrm{Lex}})$ is of the form
$$ \{ (a_1,\ldots,a_n) \in M : a_{1} = a_{2} = \ldots = a_k = 0\} $$
for $0 \leq k \leq  n$.
So in this case $(M,+,<_{\mathrm{Lex}})$ has at most $n - 1$ non-trivial convex subgroups.
Corollary~\ref{cor:lex-cor} follows from Theorem~\ref{thm:cc}, Fact~\ref{fact:dp-oag}, and Proposition~\ref{prop:weak-to-nip}.

\begin{Cor}
\label{cor:lex-cor}
Suppose $M$ is a subgroup of $(\R^n,+)$ and $\sM$ expands $(M,+,<_{\mathrm{Lex}})$.
Then $\sM$ is dp-minimal if and only if $(M,+)$ is non-singular and $\mathrm{Th}(\sM)$ is weakly $\mathrm{Th}(M,+,<_{\mathrm{Lex}})$-minimal.
If $(M,+)$ is divisible and $\sM$ is dp-minimal then $\sM$ is weakly o-minimal.
\end{Cor}

\noindent
Thus a dp-minimal expansion of $(\R^n,+,<_{\mathrm{Lex}})$ or $(\Q^n,+,<_{\mathrm{Lex}})$ is weakly o-minimal.

\medskip
\noindent 
%We say that $A \subseteq M$ is $\Z$-independent if whenever $a_1,\ldots,a_n \in A$ and $m_1,\ldots, m_n$ are integers such that $m_1 a_1 + \ldots + m_n a_n = 0$ then $m_1 = m_2 = \ldots = m_n = 0$.
%The rank of $(M,+)$ is the maximal cardinality of an independent subset of $M$.
%We include a proof of Fact~\ref{fact:finite-rank} for the sake of completeness.
We let $\rk(A)$ be the rank of an abelian group $A$, i.e. the dimension of the $\Q$-vector space $A \otimes \Q$.
We apply the well-known fact that rank is additive across short exact sequences.

\begin{Fact}
\label{fact:finite-rank}
If $\rk(M) < \aleph_0$ then $(M,+,<)$ has only finitely many convex subgroups.
\end{Fact}

\begin{proof}
The case when $M = \{0\}$ is trivial, so we suppose that $M$ is infinite.
We show that if $(M,+,<)$ has at least $n$ non-trivial convex subgroups then $\rk(M) \geq n + 1$ by applying induction on $n$.
As $M$ is torsion free $\rk(M) \geq 1$, so the case $n = 0$ is trivial.
Suppose $n \geq 1$.
Let $H$ be a non-trivial convex subgroup of $(M,+,<)$ with at least $n - 1$ convex subgroups.
By induction $\rk(H) \geq n$.
As $H$ is proper $M/H \neq \{0\}$.
So $M/H$ is torsion free, hence $\rk(M/H) \geq 1$.
We have an exact sequence
$$ 0 \longrightarrow H \longrightarrow M \longrightarrow M/H \longrightarrow 0. $$
So $ \rk(M) = \rk(H) + \rk(M/H) \geq n + 1. $
\begin{comment}
We suppose $(M,+,<)$ has at least $n$ non-trivial convex subgroups and show that the rank of $(M,+)$ is $\geq n + 1$.
Apply induction on $n$.
The case $n = 0$ is trivial so suppose $n \geq 1$.
Let $H$ be a convex subgroup of $(M,+,<)$ with at least $n - 1$ convex subgroups.
By induction the rank of $H$ is at least $n$.
Let $a_1,\ldots,a_n$ be $\Z$-independent elements of $H$.
We show that $a_1,\ldots,a_{n + 1}$ are $\Z$-independent.
Suppose $m_1,\ldots,m_{n + 1}$ are integers such that $m_1 a_1 + \ldots + m_{n + 1} a_{n + 1} = 0$ and $m_i \neq 0$ for some $i$.
As $a_1,\ldots,a_{n }$ are $\Z$-independent $m_{n + 1} \neq 0$.
So $a_{n+1} = m_{n+1}^{-1} (-m_1 a_1 - \ldots - m_{n} a_{n})$.
As $H$ is a subgroup $-m_1 a_1 + \ldots -  m_{n } a_{n}$ is an element of $H$.
As $H$ is a convex subgroup, $\{ a \in M : m_{n + 1} a \in H \}$ is a subset of $H$.
So $a_{n + 1} \in H$, contradiction.
\end{comment}
\end{proof}

\noindent Corollary~\ref{cor:finite-rank} follows from Theorem~\ref{thm:cc} and Fact~\ref{fact:finite-rank}.

\begin{Cor}
\label{cor:finite-rank}
Suppose $(M,+)$ has finite rank.
Then $\sM$ is dp-minimal if and only if $(M,+)$ is non-singular and $\mathrm{Th}(\sM)$ is weakly $\mathrm{Th}(M,+,<)$-minimal.
In particular if $(M,+)$ is divisible and $\sM$ is dp-minimal then $\sM$ is weakly o-minimal.
\end{Cor}

\begin{comment}
\noindent 
We now describe some further results on non-valuational structures.
Theorem~\ref{thm:cc} shows that if $(M,+)$ is divisible and $\sM$ is dp-minimal and non-valuational then $\sM$ is weakly o-minimal. 
It is shown in \cite[Lemma 2.6, Theorem 2.15]{Wencel1} that a non-valuational weakly o-minimal expansion of an ordered group has weakly o-minimal theory.
If the theory of $\sM$ is weakly o-minimal then $\sM$ is dp-minimal.
Corollary~\ref{cor:weak} follows.

\begin{Cor}
\label{cor:weak}
Suppose $(M,+,<)$ is divisible.
If $\mathscr{M}$ is non-valuational then $\mathscr{M}$ is dp-minimal if and only if $\mathscr{M}$ is weakly o-minimal.
It follows that if $(M,+,<)$ is divisible and archimedean then $\sM$ is dp-minimal if and only if $\sM$ is weakly o-minimal.
\end{Cor}
\end{comment}

\subsection{The discrete non-valuational case}
We prove Theorem~\ref{thm:ZZ}, a generalization of Fact~\ref{fact:Z}.
An ordered abelian group is a $\Z$-group if it is elementarily equivalent to $(\Z,+,<)$.
Model completeness of $(\Z,+,<)$ shows that if $(M,+,<)$ is a $\Z$-group then the embedding $\Z \to M$ given by $k \mapsto k1_M$ is elementary.

\begin{Thm}
\label{thm:ZZ}
Suppose $(M,+,<)$ is discrete.
Then the following are equivalent
\begin{enumerate}
\item $\sM$ is dp-minimal and non-valuational.
\item $(M,+,<)$ is a $\Z$-group and $\sM$ is interdefinable with $(M,+,<)$.
\end{enumerate}
\end{Thm}

\noindent
Theorem~\ref{thm:ZZ} requires several tools.
Standard model-theoretic results on ordered abelian groups imply that a non-valuational discrete ordered abelian group is a $\Z$-group.
We summarize these results in the proof of Fact~\ref{fact:RZ}.

\begin{Fact}
\label{fact:RZ}
Suppose $(M,+,<)$ is discrete.
Then $(M,+,<)$ is non-valuational if and only if $(M,+,<)$ is a $\Z$-group.
\end{Fact}

\begin{proof}
Any $\Z$-group is non-valuational as $(\Z,+,<)$ is non-valuational.
A regular ordered abelian group is elementairly equivalent to an archimedean ordered abelian group by~\cite{Ro-Za}.
So if $(M,+,<)$ is regular then $(M,+,<)$ is a $\Z$-group.
Suppose $(M,+,<)$ is not regular and fix $n$ such that $(M,+,<)$ is not $n$-regular.
Then $R_n$ is a non-trivial $(M,+,<)$-definable convex subgroup.
So $(M,+,<)$ is valuational.
\end{proof}

\noindent 
Fact~\ref{fact:cluckers} is due to Cluckers~\cite[Theorem 5]{Cluckers}.
It is a generalization of the theorem of Micheaux and Villemaire~\cite{MiVi-Presburger}.

\begin{Fact}
\label{fact:cluckers}
Suppose $(M,+,<)$ is a $\Z$-group and $\sM$ is sufficiently saturated.
If $\sM$ is $(M,+,<)$-minimal then $\sM$ is interdefinable with $(M,+,<)$.
\end{Fact}

\noindent
We now prove Theorem~\ref{thm:ZZ}.

\begin{proof}
It is easy to see that $(2)$ implies $(1)$.
We show that $(1)$ implies $(2)$.
Suppose $(2)$.
Then $(M,+,<)$ is a $\Z$-group by Fact~\ref{fact:RZ}.
We may pass to an elementary extension of $\sM$ without loss of generality, so suppose $\sM$ is sufficiently saturated.
By Fact~\ref{fact:cluckers} it suffices to show that $\sM$ is $(M,+,<)$-minimal.
By Proposition~\ref{prop:going-down} it suffices to show that every definable cnc subset of $M$ is $(M,+,<)$-definable.
It suffices to show that every definable convex set is $(M,+,<)$-definable.
As $\sM$ is discrete and non-valuational Lemma~\ref{lem:discrete-max} shows that every $\sM$-definable convex subset of $M$ which is bounded above has a maximum and every bounded below $\sM$-definable convex subset of $M$ has a minimum.
So every $\sM$-definable convex subset of $M$ is an interval.
\end{proof}

\subsection{An application to dp-minimal expansions of fields}
\label{section:fields}
Recall that a valuation on an ordered field is convex if its valuation ring is convex.

\begin{Cor}
\label{cor:fieldor}
Let $(K,+,\times,<)$ be an ordered field and $\mathscr{K}$ be an expansion of $(K,+,\times,<)$.
If $\mathscr{K}$ is dp-minimal then $\mathscr{K}$ is either weakly o-minimal or defines a non-trivial convex valuation.
\end{Cor}

\noindent 
If $(K,+,\times,<)$ is real closed then $(K,+,\times,<)^{\mathrm{Sh}}$ is weakly o-minimal by Fact~\ref{fact:shelah} and $(K,+,\times,<)^{\mathrm{Sh}}$ defines every convex valuation on $(K,+,\times,<)$.

\begin{proof}
If $\mathscr{K}$ is non-valuational then $\mathscr{K}$ is weakly o-minimal by Theorem~\ref{thm:cc}.
Suppose $\mathscr{K}$ is valuational and let $H$ be a non-trivial convex subgroup of the additive group of $K$.
It is easy to see that $\{ a \in K : aH \subseteq H \}$ is a non-trivial convex subring of $K$, and any convex subring of $K$ is a valuation ring.
\end{proof}

\noindent
Weakly o-minimal non-valuational structures are quite similar to o-minimal structures~\cite{Wencel1, Wencel, Keren-thesis}.
So Corollary~\ref{cor:fieldor} shows that a dp-minimal expansion of an ordered field is either ``almost o-minimal" or admits a definable valuation.
\medskip

\noindent
Suppose that $L$ is a field and $\Sa L$ is a dp-minimal expansion of $L$.
It is an open question whether one of the following necessarily holds:
\begin{enumerate}
    \item $\Sa L$ is strongly minimal,
    \item $\Sa L$ admits a definable field order, or
    \item $\Sa L$ admits a non-trivial definable valuation.
\end{enumerate}
Johnson~\cite{Johnson} has shown that if $L$ is not algebraically closed then $(2)$ or $(3)$ holds, so it suffices to show that a non strongly minimal dp-minimal expansion of an algebraically closed field admits a non-trivial definable valuation.
If the trichotomy above holds, then any dp-minimal expansion of a field is either strongly minimal, ``valuational", or ``almost o-minimal".

\section{Cyclically ordered abelian groups}
\noindent
Throughout this section $(M,+)$ is an abelian group.
A cyclic group order on $(M,+)$ is a ternary relation $C$ such that for all $a,b,c \in M$:
\begin{enumerate}
\item if $C(a,b,c)$, then $C(b,c,a)$,
\item if $C(a,b,c)$, then $C(c,b,a)$ fails,
\item if $C(a,b,c)$ and $C(a,c,d)$ then $C(a,b,d)$,
\item if $a,b,c$ are distinct, then either $C(a,b,c)$ or $C(c,b,a)$,
\item $C$ is invariant under the group operation.
\end{enumerate}
We say that a subset $J$ of $M$ is convex if whenever $a,b \in J$ are distinct we either have $\{ t : C(a,t,b) \} \subseteq J$ or $\{ t : C(b,t,a) \} \subseteq J$.
Given $a,c \in M$ we define the open interval with endpoints $a,c$ to be the set of $b \in M$ such that $C(a,b,c)$.
The collection of open intervals forms the basis for a group topology on $(M,+)$.
We say that $C$ is dense if the group topology is non-discrete.
\medskip

\noindent
\textbf{Throughout this section $C$ is a dense cyclic group order on $(M,+)$ and $\sM$ is an expansion of $(M,+,C)$}.
\medskip

\noindent
%Let $\rho$ be the quotient map $(\R,+) \to (\R/\Z,+)$.
We equip $\R/\Z$ with the cyclic group order $S$ such that whenever $a,b,c \in \R$ and $0 \leq a,b,c < 1$ then $S(a + \Z,b + \Z,c + \Z)$ holds if and only if either $a < b < c$, $b < c < a$, or $c < a < b$.
Given an injective character $\chi : M \to \R/\Z$ we let $S_\chi$ be the cyclic group order on $M$ where $S_\chi(a,b,c)$ if and only if $S(\chi(a),\chi(b),\chi(c))$.
We say that $C$ is \textbf{archimedean} if it is of this form.
One can show that $C$ is archimedean if and only if there are no $a,b \in M$ such that $C(0,na,b)$ for all $n$, see \cite{Sw-cyclic}.
\medskip

\noindent
This section is devoted to the Proof of Theorem~\ref{thm:arch-cyclic-orders}.

\begin{Thm}
\label{thm:arch-cyclic-orders}
Suppose that $(M,+,C)$ is archimedean.
Then $\sM$ is dp-minimal if and only if $(M,+)$ is non-singular and $\mathrm{Th}(\sM)$ is weakly $\mathrm{Th}(M,+,C)$-minimal.
Furthermore if $\Sa M$ is dp-minimal and $\Sa M \preceq \Sa N$ then every $\Sa N$-definable subset of $N$ is a finite union of sets of the form $a + nJ$ for convex $J \subseteq N$.
\end{Thm}

\noindent
Note that every set of the form $a + nJ$, $J \subseteq N$ convex, is $\Sh N$-definable.
We first use the universal cover described below to prove the right to left implication of Theorem~\ref{thm:arch-cyclic-orders}.

\subsection{The universal cover}
\label{section:cover}
The universal cover largely reduces the theory of cyclically ordered abelian groups to that of linearly ordered abelian groups.
A universal cover $(H,+,<,u,\pi)$ of $(M,+,C)$ consists of an ordered abelian group $(H,+,<)$, a distinguished positive element $u$ of $H$ such that $u\Z$ is cofinal in $H$, and a surjective group homomorphism $\pi : H \to M$ with kernel $u\Z$ such that for all $0 \leq a,b,c < u$ we have $C(\pi(a), \pi(b), \pi(c))$ if and only if one of the following holds : $a < b < c$, $b < c < a$, or $c < a < b$.
We call $\pi$ a covering map.
It is easy to see that $(H,+,<,u,\pi)$ is unique up to unique isomorphism.
Note $\pi$ restricts to a bijection $[0,u) \to M$.
\medskip

\noindent
Note that $(\R,+,<,1)$ is a universal cover of $(\R/\Z,+,S)$, and if $(M,+,C)$ is a subgroup of $(\R/\Z,+)$ with the induced cyclic group order, and $H$ is the preimage of $M$ under the quotient map $\R \to \R/\Z$ then $(H,+,<,1)$ is a universal cover of $(M,+,C)$.
\medskip

\noindent
We describe the standard construction of a universal cover.
Let $\prec$ be the binary relation on $M$ where $a \prec b$ if either $C(0,a,b)$ or $a = 0$ and $b \neq 0$.
It is easy to see that $\prec$ is a linear order on $M$ ($\prec$ is typically not a group order).
We let $H$ be $\Z \times M$, let $<$ be the lexicographic product of the usual order on $\Z$ and $\prec$, let $u$ be $(1,0)$, let $+$ be given by
$$(k, a)+ (k', a') = 
\begin{cases}
(k+k', a+a' ) & \text{ if } a =0  \text{ or } a' = 0 \text{ or } C(0, a, a +a') ,\\
(k+k'+1, a+a' ) & \text{ otherwise},
\end{cases}
 $$
and let $\pi : H \to M$ be the projection onto the second coordinate.
Then $(H,+,<,u)$ is a universal cover of $(M,+,C)$ with covering map $\pi$.
\medskip

\noindent
We now observe that $(M,+,C)$ is definable in $(H,+,<,u)$.
We define

$$ a \tilde{+} b =
\begin{cases}
a + b & \text{ if } a + b \in [0,u) ,\\
a + b - u & \text{ otherwise}
\end{cases}
\quad \text{for all} \quad a,b \in [0,u).
$$
We define $\tilde{C}$ by setting $\tilde{C}(a,b,c)$ for any $a,b,c \in [0,u)$ such that $a < b < c$ or $b < c < a$ or $c < a < b$.
Then $\pi : H \to M$ induces an isomorphism $([0,u),\tilde{+}, \tilde{C}) \to (M,+,C)$.
\medskip

\noindent
We let $I := (-u,u)$ and let $+_u$ be the restriction of $+$ to $I$, i.e. the ternary relation on $I$ where $a +_u b = c$ when $a,b,c \in I$ and $a + b = c$.
So $(I,+_u,<)$ is a local topological group.
We show that $(I,+_u,<)$ is definable in $(M,+,C)$.
Let $M^\geq$ be a copy of $M$ and $M^- = \{-1\} \times (M \setminus \{0\})$.
We denote an element $(-1,a)$ of $M^-$ as $-a$.
We will identify $M^\geq$ with $[0,u)$ and $M^-$ with $(-u,0)$.
We define an order $\triangleleft$ on $M^- \cup M^\geq$ in the following way
\begin{enumerate}[leftmargin=*]
    \item if $a,b \in M^\geq$ and $a \prec b$ then $a \triangleleft b$,
    \item if $-a,-b \in M^-$ and $b \prec a$ then $-a \triangleleft -b$,
    \item if $-a \in M^-$ and $b \in M^\geq$ then $-a \triangleleft b$.
\end{enumerate}
We define a ternary relation $\hat{+}$ on $M^- \cup M^\geq$ by declaring the following
\begin{enumerate}[leftmargin=*]
\item if $a,b,c \in M^- \cup M^\geq$ and $a \hat{+} b = c$ then $b \hat{+} a = c$,
\item $a \hat{+} 0 = a$ for all $a \in M^- \cup M^\geq$,
\item if $a,b,c \in M^\geq$ are non-zero then $a \hat{+} b = c$ if and only if $C(0,a,a+b)$ and $a + b = c$,
\item if $-a,-b,-c \in M^-$ then $-a \hat{+} (-b) = -c$ if and only if $C(0,a,a+b)$ and $a + b = c$,
\item if $a \in M^\geq$, $-b \in M^-$, and $c \in M^\geq$ then $a \hat{+} (-b) = c$ if and only if $b \hat{+} c = a$,
\item if $a \in M^\geq$, $a \neq 0$, $-b \in M^-$, and $-c \in M^-$ then $a \hat{+} (-b) = -c$ if and only if $a \hat{+} c = b$.
\end{enumerate}
Let $\iota : M^- \cup M^\geq \to I$ be given by declaring $\iota(a) = b$ if either $a \in M^\geq$ and $b \in [0,u)$ satisfies $\pi(b) = a$ or $a \in M^-$ and $b \in (-u,0)$ satisfies $\pi(-b) = a$.
It is easy to see that $\iota$ gives an isomorphism $(M^- \cup M^\geq, \hat{+}, \triangleleft) \to (I,+_u,<)$.
We therefore regard $(I,+_u,<)$ as an $(M,+,C)$-definable local topological group.
\medskip

\noindent
We let $\equiv_n$ be the relation of equivalence modulo $nM$ on $I$.

\begin{Lem}
\label{lem:equiv}
For all $n$, $\equiv_n$ is $(I,+_u,<)$-definable.
For all $n$ and $a \in H$, $I \cap (nH +a)$ is $(I,+_u,<)$-definable.
\end{Lem}

\begin{proof}
Note that the second claim follows from the first.
Observe that $a \in I$ is an element of $nH$ if and only if there is a $b \in I$ such that $nb = a$.
Therefore $nH \cap I$ is $(I,+_u,<)$-definable.
If $a,b \in [0,u)$ then $a - b \in I$.
So if $a,b \in I$ then $a \equiv_n b$ if and only if there is a $c \in nH \cap [0,u)$ such that $a +_u c = b$.
If $a,b \in (-u,0)$ then $a \equiv_n b$ if and only if $-a \equiv_n -b$.
If $a \in [0,u)$ and $b \in (-u,0)$ then $a \equiv_n b$ if and only if $a \equiv_n -b$.
\end{proof}

\subsection{Proof of Theorem~\ref{thm:arch-cyclic-orders}}
We continue with the notation of Section~\ref{section:cover}.
The right to left implication of Theorem~\ref{thm:arch-cyclic-orders} follows from Propositions~\ref{prop:weak-to-nip} and \ref{prop:cyclic-char}.

\begin{Prop}
\label{prop:cyclic-char}
Suppose that $(M,+,C)$ is archimedean.
Then $(M,+,C)$ is dp-minimal if and only if $(M,+)$ is non-singular.
\end{Prop}

\noindent
We will apply Lemma~\ref{lem:short-exact} to prove Proposition~\ref{prop:cyclic-char}.

\begin{Lem}
\label{lem:short-exact}
Suppose $(H,+,<,u)$ is a universal cover of $(M,+,C)$ with covering map $\pi$.
Then $(H,+)$ is non-singular if and only if $(M,+)$ is non-singular.
\end{Lem}

\noindent
Below $\otimes$ is the tensor product of abelian groups.
We use some basic facts about $\otimes$.

\begin{proof}
The left to right implication holds as a quotient of a non-singular abelian group is non-singular.
We show that $|H/nH| \leq n|M/nM|$, the other implication follows.
Let $\iota : \Z \to H$ be given by $\iota(k) := ku$.
The sequence
$$ \Z \mathop{\longrightarrow}^{\iota} H \mathop{\longrightarrow}^{\pi} M \to 0 $$
is exact, so tensoring with $\Z/n\Z$ we get an exact sequence
$$ \Z \otimes \Z/n\Z \longrightarrow H \otimes \Z/n\Z \longrightarrow M \otimes \Z/n\Z \longrightarrow 0. $$
Recall that $A \otimes \Z/n\Z$ is isomorphic to $A/nA$ for any abelian group $A$.
So we get an exact sequence
$$ \Z/n\Z \longrightarrow H/nH \longrightarrow M/nM \longrightarrow 0. $$
So $|H/nH| \leq n |M/nM|$.
\end{proof}

\noindent
We now prove Proposition~\ref{prop:cyclic-char}.

\begin{proof}
Suppose that $(M,+)$ is non-singular.
So $(H,+)$ is non-singular by Lemma~\ref{lem:short-exact} hence $(H,+,<)$ is dp-minimal by Fact~\ref{fact:dp-oag}.
The manner in which $(M,+,C)$ is defined in $(H,+,<)$ above shows that $(M,+,C)$ is dp-minimal.
Now suppose that $(M,+)$ is singular.
Fix $n$ such that $|M/nM| \geq \aleph_0$.
As $(M,+,C)$ is archimedean and dense we see that $nM$ is dense in $M$, so each coset of $nM$ is also dense in $M$.
We construct an array violating dp-minimality.
Let $(a_i : i \in \N)$ be a sequence of elements of $M$ which lie in distinct cosets of $nM$.
Let $(I_i : i \in \N)$ be a sequence of pairwise disjoint nonempty open intervals in $M$.
Then $I_i \cap (a_j + nM) \neq \emptyset$ for all $i,j$.
It follows that $(a_i + nM : i \in \N)$ and $(I_i : i \in \N)$ form an array witnessing $\dprk(M,+,C) \geq 2$.
\end{proof}

\noindent
We now prove the left to right implication of Theorem~\ref{thm:arch-cyclic-orders}.
Let $\sI$ be the structure induced on $I$ by $\sM$.
Note that $\sI$ expands $(I,+_u,<)$.
Therefore $\sM$ and $\sI$ define isomorphic copies of each other and $\sM$ is dp-minimal if and only if $\sI$ is dp-minimal.

\begin{Lem}
\label{lem:cyclic-local}
Suppose $\sI$ is dp-minimal and $X \subseteq I$ is $\sI$-definable and nonempty.
Then there is a $p \in X$ and a nonempty open interval $J \subseteq I$ containing $p$ such that $J \cap X = J \cap Y$ where $Y$ is a finite union of cosets of $nH$ for some $n$.
\end{Lem}

\begin{proof}
The lemma is trivial when $X$ is empty, so we suppose $X$ is nonempty.
Fix $a \in X$ and let $L \subseteq I$ be an open interval containing $a$ such that $L - a \subseteq I$.
Note that the map $L \to L - a$ given by $x \mapsto x - a$ is then $(I,+_u,<)$-definable. After replacing $X$ and $L$ with $(X \cap L) - a$ and $L - a$ we suppose that $0 \in X$.
After passing to an elementary extension if necessary suppose $\sI$ is sufficiently saturated.
Applying saturation we obtain  convex $C \subseteq I$ which contains a positive element and is closed under addition and additive inverse, so $(C,+_u,<)$ is an ordered abelian group.
By Fact~\ref{fact:shelah} the structure $\mathscr{C}$ induced on $C$ by $\sI$ is dp-minimal.
Applying Lemma~\ref{lem:local-behavior} to $\mathscr{C}$ and $C \cap X$ we obtain a $p \in C \cap X$ and a nonempty open interval $J \subseteq C$ such that $p \in J$ and $J \cap X = J \cap Y$ where $Y$ is a finite union of cosets of $nC$ for some $n$.
The lemma now follows easily.
\end{proof}

\noindent
Lemma~\ref{lem:cyclic-gen} now follows in the same way as Theorem~\ref{thm:gen}.
We omit the proof.

\begin{Lem}
\label{lem:cyclic-gen}
Suppose $\sI$ is dp-minimal and $X \subseteq I$ is $\sI$-definable.
Then $X$ is a finite union of sets of the form $U \cap (nH + a)$ where $U \subseteq I$ is definable and open and $a \in I$.
\end{Lem}

\begin{comment}
\noindent 
The proof is very similar to arguments in Section~\ref{section:convex} so we omit many details.

\begin{proof}
Applying Lemma~\ref{lem:cyclic-local} and compactness we obtain $n$ such that every nonempty open interval $J \subseteq I$ there is a nonempty open interval $J' \subseteq J$ such that $J' \cap X = J' \cap Y$ where $Y$ is a finite union of cosets of $nH$.
Let $A \subseteq I$ contain exactly one element from each coset of $nH$ which intersects $I$.
For each $a \in A$ let $U_a$ be the set of $b \in I$ such that there is an open interval $a \in J \subseteq I$ such that $J \cap (a + nH) \subseteq J \cap X$.
By Lemma~\ref{lem:equiv} each $U_a$ is $\sI$-definable, note that each $U_a$ is open.
It is easy to see that $X$ is $F \cup \bigcup_{a \in A} (U_a \cap (a + nH))$ for finite $F \subseteq I$.
\end{proof}
\end{comment}

\noindent
We now describe $\Sa I$-definable sets.

\begin{Lem}
\label{lem:cyclic-open}
Suppose that $(M,+,C)$ is archimedean and $\sI$ is dp-minimal.
Then every $\sI$-definable subset of $I$ is a finite union of sets of the form $J \cap (a + nH)$ for convex $J \subseteq I$.
\end{Lem}

%\noindent
%We say that a cut $C$ in $I$ is non-valuational if for every positive $a \in I$ there is a $0 < b < a$ such that $c +_u b \notin C$ for some $c \in C$.
%Note that a cut in $I$ is non-valuational if and only if its downwards closure in $H$ is a non-valuational cut in $(H,+,<)$.

\begin{proof}
By Lemma~\ref{lem:cyclic-gen} it suffices to let $U$ be an $\Sa I$-definable open subset of $I$ and show that $U$ is a finite union of convex sets.
We define a convex equivalence relation $E$ on $I$ by declaring $a,b$ to be $E$-equivalent if and only if $a,b$ lie in the same convex component of $U$.
Let $\Cal C$ be the family of downwards closures of $E$-classes.
So $\Cal C$ is an $\Sa I$-definable family of cuts.
It suffices to show that $\Cal C$ is finite.
It is easy to see that $\Cal C$ is nowhere dense.
As $(H,+,<)$ is archimedean every $C \in \Cal C$ is non-valuational.
A straightforward modification of the proof of Fact~\ref{fact:goodrick} shows that $\Cal C$ is finite.
\end{proof}

\noindent
We now prove the left to right implication of Theorem~\ref{thm:arch-cyclic-orders}.

\begin{proof}
Suppose $\sM$ is dp-minimal.
Then $(M,+)$ is non-singular by Lemma~\ref{lem:short-exact}.
We suppose that $X$ is an $\Sa M$-definable subset of $M$ and show that $X$ is a finite union of sets of the form $a + nJ$ for convex $J \subseteq M$.
The same proof goes through for elementary extensions of $\Sa M$, it easy follows that $\mathrm{Th}(\Sa M)$ is weakly $\mathrm{Th}(M,+,C)$-minimal.
\medskip

\noindent
Let $Y$ be the set of $a \in [0,u)$ such that $\pi(a) \in X$.
As $\sI$ is dp-minimal if follows by Lemma~\ref{lem:cyclic-gen}  that $Y$ is a finite union of sets of the form $J \cap (a + nM)$ for convex $J \subseteq I$.
So we may suppose that $Y$ is of this form.
As $X = \pi(Y)$ and $\pi$ is a homomorphism we have $X = \pi(a) + n\pi(J)$.
It is easy to see that $\pi(J)$ is convex, this is also shown carefully in \cite[Lemma 3.1]{cyclic-orders}.
\end{proof}

\subsection{An application to dp-minimal expansions of $(\Z,+)$}
Given irrational $\alpha \in \R$ let $\chi_\alpha$ be the character $\Z \to \R/\Z$ given by $\chi_\alpha(k) = \alpha k + \Z$.
Every injective character $\Z \to \R/\Z$ is of this form.
Let $S_\alpha$ be the cyclic group order induced on $\Z$ by $\chi_\alpha$.
It is known that any cyclic group order on $(\Z,+)$ is one of the following, see for example \cite[Proposition 2.5]{cyclic-orders}.
\begin{enumerate}
    \item $S_\alpha$ for irrational $\alpha \in \R$,
    \item $S^+$ where $S^+(i,j,k)$ if and only if $i < j <k$ or $j < k < i$ or $k < i < j$,
    \item $S^-$ where $S^-(i,j,k)$ if and only if $k < j < i$ or $j < i < k$ or $i < k < k$.
\end{enumerate}
It is easy to see that $(\Z,+,S^+)$ and $(\Z,+,S^-)$ are both interdefinable with $(\Z,+,<)$.
So Corollary~\ref{cor:cyclic-on-Z} follows from Theorem~\ref{thm:arch-cyclic-orders} and Fact~\ref{fact:Z}.

\begin{Cor}
\label{cor:cyclic-on-Z}
Let $C$ be a cyclic group order on $(\Z,+)$ and $\mathscr{Z}$ expand $(\Z,+,C)$.
Then $\mathscr{Z}$ is dp-minimal if and only if $\mathrm{Th}(\mathscr{Z})$ is weakly $\mathrm{Th}(\Z,+,C)$-minimal.
\end{Cor}

\noindent
The only other known examples of dp-minimal expansions of $(\Z,+)$ are $(\Z,+)$ and $(\Z,+,\prec_p)$ where $p$ is prime and $k \prec_p k'$ if and only if the $p$-adic valuation of $k$ is strictly less than that of $k'$.
It is a theorem of Conant and Pillay~\cite{ConantPillay} that there are no proper stable dp-minimal expansions of $(\Z,+)$.
It is natural to conjecture, analogously with the remarks in Section~\ref{section:fields}, that an unstable dp-minimal expansion of $(\Z,+)$ defines either a cyclic group order or a (perhaps generalized) valuation.
\medskip

\noindent
Let $\mathscr{Z}$ be an expansion of $(\Z,+,\prec_p)$.
It is natural to ask if $\mathscr{Z}$ is dp-minimal if and only if $\mathrm{Th}(\Sa Z)$ is weakly $\mathrm{Th}(\Z,+,\prec_p)$-minimal.
This fails.
Let $$ \pexp(\alpha) := \sum_{n = 0}^{\infty} \frac{\alpha^n}{n!} \quad \text{for all  } \alpha \in p\Z_p $$
be the $p$-adic exponential and $\mathscr{P}$ be the expansion of $(\Z,+)$ by all sets of the form $\{ (k_1,\ldots,k_n) \in \Z^n : (\pexp(pk_1),\ldots,\pexp(pk_n)) \in X \}$ for $(\Z_p,+,\times)$-definable $X \subseteq \Z^n_p$.
It will be shown in forthcoming work that $\mathscr{P}$ is a dp-minimal expansion of $(\Z,+,\prec_p)$ and
it is easy to see that $\mathscr{P}$ is not weakly $(\Z,+,\prec_p)$-minimal.

\section{Dp-minimal expansions of divisible archimedean groups}
\label{section:dense-pairs}
\noindent
Suppose that $H$ is a dense divisible subgroup of $(\R,+)$ and $\Sa H$ expands $(H,+,<)$.
In this section we show that if $\Sa H$ is dp-minimal then there is a canonical o-minimal expansion $\Sa R$ of $(\R,+,<)$ such that $\Sh H$ is interdefinable with the structure induced on $H$ by $\Sa R$.
We first describe $\Sa R$ in the case when $\Sa H$ is o-minimal.
Suppose $\Sa H$ is o-minimal.
It is a theorem of Laskowski and Steinhorn~\cite{LasStein} that $\Sa H$ is an elementary substructure of a unique o-minimal expansion $\Sa R$ of $(\R,+,<)$.
Note that the structure induced on $H$ by $\Sa R$ is a reduct of $\Sh H$.
We show that $\Sh H$ is a reduct of the induced structure.
Suppose that $X$ is an $\Sh H$-definable subset of $H^n$.
Let $\Sa N$ be a sufficiently saturated elementary extension of $\Sa R$ and $X^*$ be an $\Sa N$-definable subset of $N^n$ such that $X = X^* \cap H^n$.
By the Marker-Steinhorn theorem $\Sh R$ is interdefinable with $\Sa R$, hence $X^* \cap \R^n$ is $\Sa R$-definable.
Thus $X = (X^* \cap \R^n) \cap H^n$ is definable in the structure induced on $H$ by $\Sa R$.
\medskip

\noindent 
We first gather some background.
An $\nip$ structure $\sM$ is \textbf{Shelah complete} if $\Sh M$ is interdefinable with $\Sa M$.
It follows easily from Fact~\ref{fact:shelah} that $\Sh M$ is Shelah complete when $\Sa M$ is $\nip$, so an $\nip$ structure is Shelah complete if and only if it is interdefinable with the Shelah expansion of some structure.
Fact~\ref{fact:ms} is a special case of the Marker-Steinhorn theorem~\cite{MaSt-definable}.

\begin{Fact}
\label{fact:ms}
Any o-minimal expansion of $(\R,<)$ is Shelah complete.
\end{Fact}

\noindent
The structure \textbf{induced} on $B \subseteq M^n$ by $\Sa M$ is the structure on $B$ with an $m$-ary predicate defining $B^m \cap X$ for every $\mathscr{M}$-definable $X \subseteq M^{mn}$.
If $\Sa M$ expands a linear order then the \textbf{open core} $\Sa M^\circ$ of $\Sa M$ is the reduct of $\Sa M$ generated by all closed $\Sa M$-definable sets.
Suppose that $\Sa M$ is an o-minimal expansion of an ordered abelian group $(M,+,<)$ and $H$ is a dense divisible subgroup of $(M,+)$.
We say that $(\Sa M,H)$ is a \textbf{dense pair} if for all $(\Sa M,H) \preceq (\Sa N,K)$ we have
\begin{enumerate}[label=(A\arabic*)]
\item $(\Sa M,H)$ is $\nip$,
\item $(\Sa N,K)^\circ$ is interdefinable with $\Sa N$, and
\item every $(\Sa N,K)$-definable subset of $K^n$ is of the form $K^n \cap X$ for an $\Sa N$-definable subset $X$ of $N^n$.
\end{enumerate}
We also use the term ``dense pair" in slight variations of the situation above.
What we mean in these cases should be clear to the reader.
By $($A$3)$ the theory of the structure induced on $H$ by $(\Sa M,H)$ is weakly o-minimal and therefore dp-minimal.
\medskip

\noindent
Suppose that $H$ is a dense divisible subgroup of $(\R,+,<)$.
In this section we describe a canonical correspondence between Shelah complete dp-minimal expansions of $(H,+,<)$ and dense pairs $(\Sa R,H)$.
We show that if $(\Sa R,H)$ is a dense pair then the structure induced on $H$ by $(\Sa R,H)$ is Shelah complete and that for every Shelah complete dp-minimal expansion $\Sa H$ of $(H,+,<)$ there is a canonical o-minimal expansion $\Sa R$ of $(\R,+,<)$ such that $\Sa H$ is interdefinable with the structure induced on $H$ by $(\Sa R,H)$.
%This is an application of results from weak o-minimality which we now describe.

\subsection{Weakly o-minimal non-valuational structures}
Suppose that $(M,+,<)$ is a divisible ordered abelian group and $\sM$ is a weakly o-minimal non-valuational expansion of $(M,+,<)$.
Let $\overline{M}$ be the collection of all nonempty $\sM$-definable cuts in $M$ which are bounded above.
(Recall our convention that every cut in $M$ either does not have a supremum or is of the form $(-\infty,a]$ for some $a \in M$).
Let $\pmb{\Sa M}$ be the structure on $\overline{M}$ with an $n$-ary predicate defining the closure in $\overline{M}^n$ of every $\mathscr{M}$-definable subset of $M^n$.
The natural addition on cuts makes $\pmb{\sM}$ an expansion of an divisible ordered abelian group.
The completion $\pmb{\sM}$ is studied in \cite{Wencel1, Wencel, Keren-thesis}.
%In particular $\pmb{\sM}$ is an o-minimal expansion of an ordered abelian group.
The expansion $(\pmb{\sM},M)$ of $\pmb{\sM}$ by a unary predicate defining $M$ is studied in \cite{BaHa-a}.

\begin{Fact}
\label{fact:dense-pair}
The completion $\pmb{\Sa M}$ is o-minimal,
$(\pmb{\Sa M},M)$ is a dense pair, and the structure induced on $M$ by $(\pmb{\sM},M)$ is interdefinable with $\sM$.
\end{Fact}

\noindent
The first claim of Fact~\ref{fact:dense-pair} is shown in \cite{Keren-thesis}.
The other claims follow from \cite{BaHa-a}.
It is not explicitly shown in \cite{BaHa-a} that $(\pmb{\Sa M},M)$ is $\nip$, but this follows immediately from the results of that paper and \cite{ChSi-externally}.

\subsection{The correspondence}
Suppose that $H$ is a dense divisible subgroup of $(\R,+)$.
We describe a canonical correspondence (up to interdefinability) between
\begin{enumerate}
    \item Shelah complete dp-minimal expansions of $(H,+,<)$, and
    \item o-minimal expansions $\Sa R$ of $(\R,+,<)$ such that $(\Sa R,H)$ is a dense pair.
\end{enumerate}
Lemma~\ref{lem:shelah-complete} gives one direction of this correspondence.

\begin{Lem}
\label{lem:shelah-complete}
Suppose that $\Sa R$ is an o-minimal expansion of $(\R,+,<)$ and $(\Sa R,H)$ is a dense pair.
Then the structure $\Sa H$ induced on $H$ by $(\Sa R,H)$ is Shelah complete.
\end{Lem}

%\noindent
%Note that $(\Sa R,\R)$ is a dense pair. Lemma~\ref{lem:shelah-complete} in particular implies that an o-minimal expansion of $(\R,+,<)$ is Shelah complete.
%This is a special case of Fact~\ref{fact:ms}.

\begin{proof}
Let $(\Sa N,K)$ be a highly saturated elementary extension of $(\Sa R,H)$ and $Y$ be an $(\Sa N,K)$-definable subset of $K^n$.
We show that $H^n \cap Y$ is $(\Sa R,H)$-definable.
By $($A$3)$ there is an $\Sa N$-definable subset $X$ of $N^n$ such that $Y = K^n \cap X$.
By Fact~\ref{fact:ms} $\R^n \cap X$ is $\Sa R$-definable.
So $H^n \cap X = H^n \cap (\R^n \cap X)$ is $(\Sa R,H)$-definable.
\end{proof}

\noindent
We now describe the other direction of the correspondence.
Let $\Sa H$ be a Shelah complete dp-minimal expansion of $(H,+,<)$.
By Corollary~\ref{cor:lex-cor} $\Sa H$ is weakly o-minimal so Fact~\ref{fact:dense-pair} shows that $\pmb{\Sa H}$ is an o-minimal expansion of $(\R,+,<)$ and $(\pmb{\Sa H},H)$ is a dense pair.
By Fact~\ref{fact:dense-pair} the structure induced on $H$ by $(\pmb{\Sa H},H)$ is interdefinable with $\Sa H$.
So we can recover $\Sa H$ up to interdefinibility from $\pmb{\Sa H}$.

\begin{Lem}
\label{lem:canonical}
Suppose that $\Sa R$ is an o-minimal expansion of $(\R,+,<)$ and $(\Sa R,H)$ is a dense pair.
Let $\Sa H$ be the structure induced on $H$ by $(\Sa R,H)$.
Then $\pmb{\Sa H}$ is interdefinable with $\Sa R$.
\end{Lem}

\noindent
Our proof follows that of \cite[Proposition 3.4]{W-shelah-expansions}.
Recall that an open subset of a topological space is \textbf{regular} if it is the interior of its closure.

\begin{proof}
By \cite[Lemma 3.3]{W-shelah-expansions} every $\Sa R$-definable subset of $\R^n$ is a boolean combination of regular open $\Sa R$-definable sets.
So it suffices to suppose that $U$ is a regular open $\Sa R$-definable subset of $\R^n$ and show that $U$ is $\pmb{\Sa H}$-definable.
Note that $H^n \cap U$ is $\Sa H$-definable, so the closure of $H^n \cap U$ in $\R^n$ is $\pmb{\Sa H}$-definable.
Note that the closure of $H^n \cap U$ agrees with the closure of $U$ as $H^n$ is dense.
So by regularity $U$ is the interior of the closure of $H^n \cap U$ in $\R^n$.
So $U$ is definable in $\pmb{\Sa H}$.
\end{proof}

\subsection{Expansions of $(\Q,+,<)$}
We are particularly interested in the case $H = \Q$.

\begin{Ques}\label{ques:omin}
For which o-minimal expansions $\mathscr{R}$ of $(\R,+,<)$ is $(\mathscr{R},\mathbb{Q})$ a dense pair?
\end{Ques}

\begin{comment}
\noindent
We describe some examples.
It is well-known that $(\R,+,\times,\mathbb{Q})$ is bi-interpretable with second order arithmetic.
Fix a real number $t > 1$ and let $t^{\mathbb{Q}} := \{ t^q : q \in \mathbb{Q} \}$.
Then $\log_t : (\R^{>0},\times,t^{\mathbb{Q}}) \to (\R,+,<,\mathbb{Q})$ is an isomorphism.
Question~\ref{ques:omin} is equivalent to the following: for which o-minimal expansions $\mathscr{R}$ of $(\R^{>0},\times)$ is $(\mathscr{R},t^{\mathbb{Q}})$ $\mathrm{NIP}$?
The expansion $(\R,+,\times,t^{\mathbb{Q}})$ is $\mathrm{NIP}$ and the induced structure on $t^\Q$ is weakly o-minimal \cite{GuHi-Dependent}.
Let $C^\infty([0,1])$ be the space of smooth functions $[0,1] \to \R$ equipped with the Polish topology given by the usual family of seminorms $f \mapsto \max\{ f^{(n)}(x) : 0 \leq x \leq 1\}$.
Le Gal~\cite{LGal} has shown that there is a dense $G_\delta$ subset $Z$ of $C^\infty([0,1])$ such that $(\R,+,\times,f)$ is o-minimal for any $f \in Z$.
It seems likely that there is a comeager subset $Z'$ of $C^\infty([0,1])$ such that $(\R,+,\times,f,t^\Q)$ is $\mathrm{NIP}$ and the induced structure on $t^\Q$ is weakly o-minimal for any $f \in Z'$.
\end{comment}

\noindent 
The classification of o-minimal expansions of $(\R,+,<)$ breaks Question~\ref{ques:omin} into several cases.
Fact~\ref{fact:ps} is a special case of the Peterzil-Starchenko trichotomy theorem \cite{PS-Tri}.
We let $\R_{\mathrm{Vec}}$ be the ordered vector space $(\R,+,<,( x \mapsto \lambda x)_{\lambda \in \R})$.

\begin{Fact}
\label{fact:ps}
Suppose $\sR$ is an o-minimal expansion of $(\R,+,<)$.
Then exactly one of the following holds:
\begin{enumerate}
    \item  $\sR$ is a reduct of $\R_{\mathrm{Vec}}$,
    \item there is a nonempty open interval $I$ and $\sR$-definable $\oplus,\otimes : I^2 \to I$ such that $(I,\oplus,\otimes,<)$ is isomorphic to $(\R,+,\times,<)$.
\end{enumerate}
\end{Fact}

\noindent
It is shown in \cite{GoHi-Pairs} that $(\R_{\mathrm{Vec}},\mathbb{Q})$ is a dense pair.
So we consider the case when $(2)$ above holds.
Fact~\ref{fact:edmundo} is a special case of a theorem of Edmundo~\cite{ed-str}.

\begin{Fact}
\label{fact:edmundo}
Suppose $\sR$ is an o-minimal expansion of $(\R,+,<)$.
Then exactly one of the following holds:
\begin{enumerate}
\item There is a subfield $K$ of $(\R,+,\times)$ and a collection $\Cal B$ of bounded subsets of Euclidean space such that $(\R,+,<,\Cal B)$ is o-minimal and $\sR$ is interdefinable with $(\R,+,<,\Cal B, (x \mapsto \lambda x)_{\lambda \in K})$.
\item There are $\sR$-definable $\oplus,\otimes : \R^2 \to \R$ such that $(\R,\oplus,\otimes,<)$ is isomorphic to $(\R,+,\times,<)$.
\end{enumerate}
\end{Fact}

\subsection{Global field structure}
We discuss case $(2)$ of Fact~\ref{fact:edmundo}.
Let $\iota$ be the unique isomorphism $(\R,\oplus,\otimes,<) \to (\R,+,\times,<)$, $\sR'$ be the pushforward of $\sR$ by $\iota$, $\boxplus$ be the pushforward of $+$ by $\iota$, and $Q := \iota(\Q)$.
So $\sR'$ is an expansion of $(\R,+,\times)$ and $\iota$ gives an isomorphism $(\sR,\Q) \to (\sR',Q)$.
\medskip

\noindent
The \textit{o-minimal two group question} asks if there is necessarily either an $\sR'$-definable isomorphism $(\R,\boxplus,<) \to (\R,+,<)$ or $(\R,\boxplus,<) \to (\R_{>0},\times)$.
Thisquestion is open, but it holds for all known o-minimal expansions of $(\R,+,\times)$, see \cite{two-group}.
\medskip

\noindent
Suppose there is an $\sR'$-definable isomorphism $\tau : (\R,\boxplus,<) \to (\R,+,<)$.
Then there is $t \in \R_{>0}$ such that $(\tau \circ \iota)(a) = ta$ for all $a \in \R$, hence $\tau(Q) = t\Q$.
It follows that $(\sR',Q)$ defines $\Q$ and is therefore totally wild.
\medskip

\noindent
Suppose there is an $\sR'$-definable isomorphism $\tau : (\R,\boxplus,<) \to (\R_{>0},\times)$.
Then there is $t \in \R_{>0}$ such that $(\tau \circ \iota)(a)= t^a$ for all $a \in \R$, hence $\tau(Q) = t^\Q := \{t^q : q \in \Q \}$.
This leads to Question~\ref{ques:powers-weak}.

\begin{Ques}
\label{ques:powers-weak}
Fix $t \in \R_{>0}$.
For which o-minimal expansions $\mathscr{S}$ of $(\R,+,\times)$ is $(\mathscr{S},t^\Q)$ a dense pair? 
\end{Ques}

\noindent
It is known that $(\R,+,\times,t^\Q)$ is a dense pair~\cite{GuHi-Dependent}.
Hieronymi has shown that there is a co-countable subset $\Lambda$ of $\R \setminus \Q$ such that if $r \in \Lambda$ and $t \in \R_{>0}$ lies in the algebraic closure of $\Q(r)$ then $(\R,+,\times,x^r, t^\Q)$ is a dense pair~\cite{H-tau}.
\medskip

\noindent
If $\mathscr{S}$ defines the exponential function then $(\mathscr{S}, t^\Q)$ defines $\Q$ and is therefore totally wild so by Miller's dichotomy~\cite{miller-avoid} if $(\mathscr{S}, t^\Q)$ is a dense pair then $\mathscr{S}$ is polynomially bounded.
It follows from \cite[Proposition 1]{tychon-convex} that if $\mathscr{S}$ defines the restriction of the exponential to some nonempty open interval than $(\mathscr{S},t^\Q)$ defines $\Q$.
So for example $(\R_{\mathrm{an}}, t^\Q)$ is totally wild.
\medskip

\noindent
%We describe a potential construction of uncountably many weakly o-minimal expansions of $(\Q,+,<)$.
Let $I$ be a closed bounded interval and $C^\infty(I)$ be the space of smooth functions $I \to \R$ equipped with the topology induced by the seminorms $f \mapsto \max 
\{ f^{(n)}(r) : r \in I \}$.
So $C^\infty(I)$ is Polish.
Le Gal has shown that there is a comeager subset $\Gamma$ of $C^\infty(I)$ such that if $f \in \Gamma$ the $(\R,+,\times,f)$ is o-minimal and polynomially bounded~\cite{LGal}.
Fix $t \in \R_{>0}$.
There should be a comeager subset $\Gamma_t$ of $C^\infty(I)$ such that if $f \in \Gamma_t$ then $(\R,+,\times,f)$ is o-minimal and $(\R,+,\times,f,t^\Q)$ is a dense pair.

\subsection{Local field structure}
We now consider the case when there is a local field structure but no global field structure, i.e. when item $(1)$ of Fact~\ref{fact:ps} and item $(2)$ of Fact~\ref{fact:edmundo} both fail.
In this case the field structure is on a bounded interval $I$ and there is no definable injection $\R \to I$, so only bounded segments of $\Q$ show up inside the field structure.
Suppose $I$ is a bounded open interval, $\oplus,\otimes : I^2 \to I$ are $\sR$-definable, and $\iota : (\R,+,\times,<) \to (I,\oplus,\otimes,<)$ is an isomorphism.
Let $\sR'$ be pullback by $\iota$ of the structure induced on $I$ by $\mathscr{R}$.
So $\sR'$ is an o-minimal expansion of $(\R,+,\times)$.
After rescaling and translating if necessary we suppose $[0,1]$ is contained in $I$.
Let
\begin{equation*}
    s \tilde{+} t = \begin{cases}
    s + t & s + t < 1 \\
    s + t - 1 & \text{otherwise.}\\
    \end{cases}
    \quad \text{for all} \quad 0 \leq s,t < 1.
\end{equation*}
\noindent 
Let $\mathbb{S}$ be the pullback of $([0,1),\tilde{+})$ by $\iota$, so $\mathbb{S}$ is an $\sR'$-definable group.
Note $\mathbb{S}$ is isomorphic to $(\R/\mathbb{Z},+)$ and $q \in [0,1)$ is rational if and only if $q$ is a torsion point of $\mathbb{S}$.
So if $(\Sa R,\Q)$ is a dense pair then the expansion of $\Sa R'$ by the torsion points of $\mathbb{S}$ is a dense pair.
%In light of Fact~\ref{fact:edmundo} it seems likely that $(\sR,\Q)$ is $\mathrm{NIP}$ if and only if $(\mathscr{I}, [0,1) \cap \Q)$ is $\mathrm{NIP}$.
This leads us to Question~\ref{ques:torsion}.

\begin{Ques}\label{ques:torsion}
If the expansion of $\Sa R'$ by the torsion points of $\mathbb{S}$ is a dense pair then must $(\Sa R,\Q)$ be a dense pair?
For which o-minimal expansions $\mathscr{S}$ of $(\R,+,\times)$ and compact, connected, one-dimensional $\mathscr{S}$-definable groups $\mathbf{T}$ is the expansion of $\mathscr{S}$ by a predicate defining the torsion points of $\mathbf{T}$ a dense pair?
\end{Ques}

\noindent
If the first part of Question~\ref{ques:torsion} has a positive answer then Question~\ref{ques:omin} largely reduces to Question~\ref{ques:powers-weak} and the second part of Question~\ref{ques:torsion}.
\medskip

\noindent 
Question~\ref{ques:torsion} is closely related to the following question of Peterzil~\cite[11.2]{Peterzil-survey}: Suppose $\mathscr{R}$ is an o-minimal expansion of a real closed field and $G$ is a definably compact definable group, what is the induced structure on the torsion points of $G$?
\medskip

\noindent 
Consider the case when $\mathbf{T}$ is the unit circle equipped with complex multiplication.
Then $T$ is the set of roots of unity.
The expansion of $(\R,+,\times)$ by a predicate defining the set $T$ of roots of unity is a dense pair.
This again follows by combining general results on preservation of $\mathrm{NIP}$ from \cite{ChSi-externally} with specific tameness results for this structure given in \cite{BeZi-The}.
The results of \cite{BeZi-The} rely on the ``Lang property" for roots of unity: if $V$ is a subvariety of $\mathbb{C}^n$ then $V \cap U^n$ is a boolean combination of sets of the form
$$ \{ (g_1,\ldots,g_n) \in U^n : g_1^{m_1} \ldots g_n^{m_n} = h \} $$
for integers $m_1,\ldots,m_n$ and $h \in \mathbb{C}^\times$.
Other positive instances of Question~\ref{ques:torsion} may depend on other instances of the ``Lang property".

\subsection{The $p$-adic case}
\label{section:p adic completion}
We finally give a partial $p$-adic analogue of the results on dp-minimal expansions of divisible archimedean ordered abelian groups given above.
We continue where we left off at the end of Section~\ref{section:discretely} and continue to work in the set up of the proof of Theorem~\ref{thm:padic-arch}.
Let $\K$ be a finite extension of $\Q_p$, $K$ be an elementary subfield of $\K$, $\Sa K$ be a dp-minimal expansion of $K$, $\Sa F$ be a highly saturated elementary extension of $\Sa K$ such that $\K$ is an elementary subfield of the underlying field $F$ of $\Sa F$, and $\pmb{\Sa K}$ be the associated expansion of $\K$ defined as above.

\begin{Prop}
\label{prop:p adic completion}
The structure induced on $K$ by $\pmb{\Sa K}$ is interdefinable with $\Sh K$.
\end{Prop}

\noindent
Thus any Shelah complete dp-minimal expansion of $K$ is the structure induced on $K$ by some dp-minimal expansion of $\K$.
We first recall the necessary facts.
Fact~\ref{fact:external} is immediate from the definitions and left to the reader.

\begin{Fact}
\label{fact:external}
Suppose that $\Sa M \prec \Sa N$ are structures and $X \subseteq N^n$ is externally definable in $\Sa N$.
Then $X \cap M^n$ is externally definable in $\Sa M$.
\end{Fact}

\noindent
Fact~\ref{fact:external 1} is one way of phrasing existence of honest definitions, see \cite{ChSi-externally}.

\begin{Fact}
\label{fact:external 1}
Suppose that $\Sa M$ is $\nip$, $\Sa M \prec \Sa N$ is sufficiently saturated, and $X \subseteq M^n$ is externally definable in $\Sa M$.
Then there is a $\Sa N$-definable $D \subseteq N^n$ such that $X = D \cap M^n$ and $D \cap Z^* = \emptyset$ whenever $Z$ is an $\Sa M$-definable subset of $M^n$ such that $X \cap Z = \emptyset$ and $Z^*$ is the subset of $N^n$ defined by the same formula as $Z$.
\end{Fact}

\noindent
Fact~\ref{fact:def boolean} is a special case of the results of \cite{SW-tametop}.

\begin{Fact}
\label{fact:def boolean}
Suppose that $(E,v)$ is a valued field and $\Sa E$ is a dp-minimal expansion of $(E,v)$.
Then every $\Sa E$-definable subset of $E^n$ is a finite boolean combination of closed $\Sa E$-definable subsets of $E^n$.
\end{Fact}

\noindent
We now prove Proposition~\ref{prop:p adic completion}.
We let $\st : W^n \to \K^n$ be the coordinate-wise standard part map given by $\st(x_1,\ldots,x_n) = (\st(x_1),\ldots,\st(x_n))$.

\begin{proof}
We first show that the structure induced on $K$ by $\pmb{\Sa K}$ is a reduct of $\Sh K$.
It suffices to assume that $Y \subseteq \K^n$ is $\pmb{\Sa K}$-definable and show that $X = Y \cap K^n$ is $\Sh K$-definable.
By dp-minimality of $\pmb{\Sa K}$ and Fact~\ref{fact:def boolean} we may suppose that $Y$ is closed.
Let $Y^* = \st^{-1}(Y)$.
As $Y$ is closed we have $X = Y^* \cap K^n$.
Then $Y^*$ is definable in $\Sh F$, hence externally definable in $\Sa F$.
Hence by Fact~\ref{fact:external} $Y$ is externally definable in $\Sa K$.
\medskip

\noindent
We now suppose that $X \subseteq K^n$ is definable in $\Sh K$ and show that $X$ is definable in the structure induced on $K$ by $\pmb{\Sa K}$.
By dp-minimality of $\Sh K$ and Fact~\ref{fact:external} we may suppose that $X$ is closed in $K^n$.
Let $X'$ be the closure of $X$ in $\K^n$.
As $X$ is closed we have $X = X' \cap K^n$, so it is enough to show that $X'$ is definable in $\pmb{\Sa K}$.
We let $D$ be the $\Sa F$-definable set provided by Fact~\ref{fact:external 1}.
We show that $\st(D \cap W^n) = X'$.
An easy saturation argument shows that $\st(D \cap W^n)$ is closed, and $X$ is contained in $D \cap W^n$, so $X'$ is contained in $\st(D \cap W^n)$.
We show that $\st(D \cap W^n)$ is contained in $X'$.
Fix $p \notin X'$.
As $X'$ is closed there is a definable open neighbourhood $U$ of $p$ such that $U \cap X' = \emptyset$.
We may suppose that $U = U_1 \times \ldots \times U_n$ for balls $U_1,\ldots,U_n$ in $\K$.
Let $U^*$ be the subset of $F^n$ defined by the same formula as $U$.
As $U$ is clopen we have $U^* = \st^{-1}(U)$.
By Fact~\ref{fact:external 1} $D$ is disjoint from $U^*$.
Hence $\st(D \cap W^n)$ does not intersect $U$, so $p \notin \st(D \cap W^n)$.
\end{proof}

\bibliographystyle{abbrv}
\bibliography{Ref}
\end{document}

%% file: makros.tex
\newtheorem{Th}{Theorem}[section]
\newtheorem{Thm}[Th]{Theorem}

\newtheorem{Cor}[Th]{Corollary}
\newtheorem{Prop}[Th]{Proposition}

\newtheorem{Lem}[Th]{Lemma}
\newtheorem{Fact}[Th]{Fact}
\newtheorem{Conj}[Th]{Conjecture}

\newtheorem{Ques}{Question}